\documentclass[11pt]{amsart}

\numberwithin{equation}{section}
\usepackage[a4paper, left=3.5cm, right=3.5cm]{geometry}
\usepackage{xypic}
\usepackage[english]{babel}
\usepackage{amsmath}
\usepackage{amsfonts}
\usepackage{amssymb}
\usepackage{amscd}
\usepackage{amsthm}
\usepackage{color}
\usepackage{xcolor}
\usepackage{verbatim}
\usepackage[colorlinks, linktocpage, citecolor = blue, linkcolor = blue]{hyperref}
\usepackage{tikz}	

\usetikzlibrary{matrix}					%geometric/algebraic description.
\usetikzlibrary{arrows, shapes, snakes, 
		       automata, backgrounds, 
		       petri, topaths}				
\newtheorem{theorem}{Theorem}[section]

\newtheorem{lemma}[theorem]{Lemma}
\newtheorem{proposition}[theorem]{Proposition}
\newtheorem{corollary}[theorem]{Corollary}

\theoremstyle{definition}
\newtheorem{definition}[theorem]{Definition}
\newtheorem{example}[theorem]{Example}

\newtheorem{remark}[theorem]{Remark}
\newtheorem*{acknowledgments}{Acknowledgments}
\newtheorem*{notation}{Notations and conventions}

\theoremstyle{remark}

\usepackage{amscd, amssymb}
\newcommand\mynote[1]{\marginpar{\ \\ \small \tt #1}}
\newcommand\bel[1]{{\mynote{#1}}\begin{equation}\label{#1}}
\newcommand\mylabel[1]{\label{#1}}

\newcommand{\ZZ}{\mathbb{Z}}
\newcommand{\QQ}{\mathbb{Q}}

\newcommand{\CC}{\mathbb{C}}

\newcommand{\FF}{\mathbb{F}}

\newcommand{\GG}{\mathbb{G}}

\newcommand  {\shA}     {\mathcal{A}}

\newcommand  {\shE}     {\mathcal{E}}
\newcommand  {\shF}     {\mathcal{F}}
\newcommand  {\shG}     {\mathcal{G}}

\newcommand  {\shM}     {\mathcal{M}}

\newcommand  {\shX}     {\mathcal{X}}

\newcommand  {\shO}     {\mathcal{O}}

\newcommand  {\Aut}     {\operatorname{Aut}}

\newcommand  {\Br}      {\operatorname{Br}}

\newcommand  {\et}      {{\text{\rm \'{e}t}}}

\newcommand  {\Ext}     {\operatorname{Ext}}

\newcommand  {\fet}      {{\text{\rm f\'{e}t}}}

\newcommand  {\Gal}     {\operatorname{Gal}}
\newcommand  {\GL}      {\operatorname{GL}}
\newcommand  {\Gr}      {\operatorname{Gr}}

\newcommand  {\Hom}     {\operatorname{Hom}}

\newcommand  {\lra}     {\longrightarrow}

\newcommand  {\Pic}     {\operatorname{Pic}}
\newcommand  {\PGL}     {\operatorname{PGL}}

\newcommand  {\PSL}     {\operatorname{PSL}}

\newcommand  {\ra}      {\rightarrow}
\newcommand  {\Ra}      {\Rightarrow}

\newcommand  {\Spec}    {\operatorname{Spec}}

\newcommand  {\Tor}     {\operatorname{Tor}}

\renewcommand  {\top}     {\operatorname{top}}

\newcommand {\zar}      {{\operatorname{zar}}}

\def\mydate{\number\day\space\ifcase\month \or January\or February\or March\or
April\or May\or June\or July\or
August\or September\or October\or November\or December\fi \space\number\year}
\begin{document}

\title[ Brauer groups and \'etale homotopy type]
      { Brauer groups and \'etale homotopy type}

\author[Mohammed Moutand]{Mohammed Moutand}
\address{ Moulay Ismail University, 
 Department of mathematics, 
Faculty of sciences,  Mekn\`es,  B.P. 11201 Zitoune,  Mekn\`es,  Morocco.}
\curraddr{}
\email{m.moutand@edu.umi.ac.ma}
%\subjclass[2010]{14F22, 14F35}

%\subjclass{14F22,  20J06,  32J15}

%\dedicatory{Final version,  29 January 2020} 

\begin{abstract}
Extending a result of Schr\"oer  on a Grothendieck question in the context of   complex analytic spaces,  we prove that the surjectivity of the Brauer map $\delta: \Br(X) \ra  H_{\et}^2(X, \GG_{m, X})_{\rm tor}$ for algebraic schemes depends on their \'etale homotopy type. We use properties of algebraic $K(\pi, 1)$ spaces to apply this to some  classes of proper and smooth algebraic schemes. In particular we recover a  result of Hoobler and Berkovich for abelian varieties. Further, we  give an additional condition for the surjectivity of $\delta$ which involves pro-universal covers. All proposed conditions turn out to be equivalent for smooth quasi-projective varieties. 
\end{abstract}

\maketitle
\tableofcontents

\section*{Introduction}
\mylabel{introduction}

In \cite{GR1} Grothendieck established a general formalism for the theory of Azumaya algebras,  which allows to construct the Brauer group $\Br(X)$  of a scheme (or more generally of a locally ringed topos),  and hence generalizing the previous construction of Azumaya for local rings  and that of Auslander-Goldman for   arbitrary  commutative rings. He defined $\Br(X)$ as the set of classes of  Azumaya algebras on $X$ modulo Morita equivalence,  or equivalently the set of equivalence  classes of principal $\PGL_n$-bundles. In a part of his works, he constructed  via non abelian cohomology  an injective homomorphism of groups $\delta: \Br(X) \ra  \Br'(X)$ called the Brauer map (see Theorem \ref{groth}),  where $\Br'(X):= H_{\et}^2(X, \GG_{m, X})_{\rm tor}$  is  the torsion part of   the cohomological  Brauer group $ H_{\et}^2(X, \GG_{m, X})$,  and asked in which case this map is a bijection,  in other words, for a given scheme $X$,  does any  cohomological   Brauer class $\beta \in \Br'(X) $ comes from an Azumaya algebra ?. When  $X$ is a complex analytic space endowed with the sheaf of holomorphic functions,  one can define by a similar construction the analytic Brauer group $\Br(X)$ of $X$,  and hence we get in terms of cohomology of sheaves a well defined injective Brauer map $\delta: \Br(X) \ra  \Br'(X):=H^2(X, \GG_{m, X})_{\rm tor}$ (cf. \cite{HS}, \cite{SCH3}).

A positive answer to this question for any class of schemes will be with a big interest when it comes to  the computation of $\Br(X)$, this is due to the fact that the cohomological Brauer group $\Br'(X)$ appears  in many fundamental exact sequences with various other cohomology groups (eg. Kummer sequence, Artin-Schreier sequence, exponential exact sequence,...). The question is also partially related to the problem of determining weather an algebraic stack is a quotient stack (see \cite{EHKV}).

The Brauer map is known to be surjective for the following classes of algebraic schemes:

\begin{itemize}
\item  Regular schemes of dimension $\leq 2$: Grothendieck \cite{GR1}.
\item Abelian varieties: Berkovich \cite{BERK},   and more generally abelian schemes: Hoobler \cite{HOO1}.
\item Character free algebraic groups: Iversen \cite{IVR}.
\item Affine schemes,  and  separated unions of two affine schemes: Gabber \cite{GBBR}. Simplified proofs were given by Hoobler \cite{HOO2} with more additional results.
\item Schemes with ample invertible sheaf: Proved by Gabber (unpublished). An alternative  proof was given  by De Jong \cite{DJNG}.
\item Separated geometrically normal algebraic surfaces: Schr\"oer   \cite{SCH1}.
\end{itemize}

For  complex spaces, we have the following treated cases:

\begin{itemize}
\item Complex torus: Elencwajg and Narasimhan  \cite{EN}.
\item Analytic K3 surfaces, Ricci-flat compact K\"{a}hler surfaces:  Huybrechts and \ Schr\"oer  \cite{HS}.

\item  Hopf manifolds,  complex lie groups and elliptic surfaces: Schr\"oer \cite{SCH3}. These are  particular cases of a general statement (see Theorem \ref{schr}) for complex analytic spaces proved by the author in $loc. cit.$ via  homotopy theory.
\end{itemize}

The equality  $ \Br(X) = \Br'(X)$ does not hold  in general. Indeed, an example of a non separated normal surface for which  $\Br(X) \neq  \Br'(X)$ was constructed in \cite{EHKV} by arguments from quotient stacks theory. The same example was treated by Bertuccioni \cite{BERTU} by means of Mayer-Vietoris sequence with a  K-theory approach.

More general variants of this problem have been established by several  authors. In  \cite{MTH1} Mathur showed   via the resolution propriety that $ \Br(\shX) = \Br'(\shX)$ when $\shX$  is a tame Artin stack of dimension $\leq 2$,  and more recently \cite{MTH2} he treated the case of algebraic spaces obtained from quasi-projective schemes by contracting a curve.   Bertolin and Galluzzi  \cite{BERT}  extended the notion of Azumaya algebras to ( non necessary algebraic ) stacks,  and as an application they gave an affirmative answer to Grothendieck  question for 1-motives $ M = [u: X \ra  G]$  defined  over noetherian schemes.  A derived variant of this question was  studied by   Toen \cite{TOE}  via the notion of derived Azumaya algebras. Extending this construction,  Antieau and Gepner \cite{ANT}  treated the problem in the context of spectral geometry. And more recently  Chough \cite{CHO}  proved similar results for algebraic stacks in  both derived and spectral contexts.

For a given  cohomological Brauer class $\beta \in H_{\et}^2(X, \GG_{m, X})_{\rm tor}$,  it is  difficult to find explicitly an Azumaya algebra on $X$ whose image under $\delta$  is $\beta$. However, many tools have been introduced to ensure the existence of the required  algebra; In \cite{LIE} Lieblich proved that for a nice scheme for which \'etale cohomology can be computed in \v Cech terms,  the class $\beta$ lies in $\Br(X)$ if only if there exists a finite locally free $\alpha$-twisted \'etale sheaf on $X$  of positive rank (see also \cite{DJNG}). Using this fact, he recovered Grothendieck and Gabber results by simplifications of Hoobler arguments in \cite{HOO2}. Another important  tool  is  the geometric interpretation  of  the cohomology groups   $H_{\et}^1(X,   \PGL_n(\mathcal{O}_X))$  and  $H_{\et}^2(X, \mathbb{G}_{m, X})$ via $\PGL_n$-torsors and $\GG_{m, X}$-gerbes. More precisely,  one can associate to any Azumaya algebra $\shA$ a  $\GG_{m,X}$-gerbe       $\shG_{\shA}$  and  a         $\PGL_n$-torsor  $P_{\shA}$ such that  the class  $[\shG_{\shA}] \in H_{\et}^2(X, \mathbb{G}_{m,X})$  is equal to the image of the class $[P_{\shA}] \in H_{\et}^1(X,   \PGL_n(\mathcal{O}_X) $  under the  boundary   map $\delta_n:  H_{\et}^1(X,   \PGL_n(\mathcal{O}_X))   \ra  H_{\et}^2(X, \mathbb{G}_{m, X})$ (see \cite[Chapter 12]{OLSS}). In the light of  this interpretation,  authors in \cite{EHKV} showed that  the class $\beta$ lies in $\Br(X)$ if only if the $\GG_{m, X}$-gerbe $\mathcal{X}_{\beta}$ associated to ${\beta}$  is a quotient stack. A useful   technical tool   used by Gabber,  Hoobler and Berkovich   which we shall adopt in this paper  states that  if there exists a finite \'etale cover (or a Galois cover) $\pi: Y \ra X$   trivializing the class $\beta$ in $H_{\et}^2(Y, \GG_{m, Y})$ then $\Br(X) = \Br'(X)$ (see Lemmas  ~\ref{etale} and   ~\ref{galois}). When $X$ is a  complex analytic space,   \ Schr\"oer   \cite{SCH3} proved that such a cover can be obtained, and hence one has $\Br(X) = \Br'(X)$,  if the topological fundamental group $\pi_1(X)$ is a good group, 
and  the  subgroup of $\pi_1(X)$-invariants inside
the Pontryagin dual $\Hom(\pi_2(X), \QQ/\ZZ)$
is trivial.  

 The aim of this paper is to extend \ Schr\"oer result to the algebraic setting. In this context, for a pointed connected noetherian scheme $(X,\bar x)$, we are going to work with  the Grothendieck   \'etale fundamental group $\pi_1^\et(X,  \bar x) $ introduced in \cite{SGA1}, and   the higher \'etale  homotopy groups $\pi_n^\et(X,  \bar x)$   ($n\geq 0$)  as defined by Artin and Masur in \cite{AM}. For our purpose, since $\pi_1^\et(X,  \bar x) $ is always profinite,  we  just have to deal with  the higher groups $\pi_n^\et(X,  \bar x)$   ($n\geq 2$). We adabt an algebraic version of \ Schr\"oer argument to  prove -by means of  Galois-Grothendieck theory- the following  main result:

\begin{theorem}{\rm (Theorem \ref {rcnb})}
Let $X$ be a regular  connected  scheme of finite type over a field k of characteristic 0, with a geometric base point $\bar x \ra X$,   such that $ \pi_2^\et(X, \bar{x})=0$. Then $\Br(X) = \Br'(X)$.
\end{theorem}
As in the topological context,  the calculation of  the higher  \'etale homotopy groups $ \pi_n^\et(X, \bar{x})$   is in general much more difficult. However,  if $X$ is in particular a geometrically unibranch scheme  with $ \pi_n^\et(X, \bar{x})=0$ for all $n \geq 2$, then this is equivalent to say that   $X$  is an algebraic $K(\pi,1)$ space (Definition \ref{defkp}). This class of spaces was largely studied by  Achinger  \cite{ACH1},\cite{ACH2}   in addition to   other variants (logarithmic \cite{ACH1}, rigid analytic and mixed characteristic \cite{ACH2}). Some properties of this class sketched  in $loc. cit.
$ will serve to get the following partial result concerning schemes over $\CC$.
\begin{corollary}
Let  $X$ be  a   smooth connected scheme of finite type  over $\CC$,  if X is an Artin neighborhood over $\Spec(\CC)$, then   $\Br(X)= \Br'(X)$.
\end{corollary}

When it comes to schemes over subfields of $\CC$, or algebraically closed fields in the proper case,  the \'etale fundamental group behaves nicely with base change (see \cite{ES}). Using   descent arguments, this can  be employed  together with   properties of algebraic $K(\pi,1)$ spaces to prove the following result for  proper and smooth schemes.

\begin{theorem} Let  $X$ be a  geometrically connected  scheme of finite type over a field k. Suppose that $k$ can be embedded as a subfield of $\CC$
,  and such that  $X_\CC$ is an Artin neighborhood over $\Spec(\CC)$. Then $\Br(X)= \Br'(X)$ in the following cases
\begin{itemize}
 \item [(i)] {\rm (Proposition \ref {prpr})} $X$  proper,  with $k$ algebraically closed   and the natural morphism $\Br(X) \lra  \Br(X_{\CC})$ is surjective. 
 \item [(ii)] {\rm (Proposition \ref {smooth})} $X$  smooth, with $k$  finitely generated  over $\QQ$.    
 \item [(iii)] {\rm (Proposition \ref {prprreg})}  $X$ regular,  proper,  with $k$ algebraically closed. 
\end{itemize}
\end{theorem}
The choice of $X_\CC$ to be  an Artin neighborhood follows from the fact that  $X_\CC(\CC)$ is a topological $K(\pi,  1)$ space with good topological fundamental group $\pi_1(X_\CC(\CC))$ (Lemma \ref {khtp}), which means that $X_\CC(\CC)$ verifies \ Schr\"oer conditions. We review these notions in Section \S \ref{S2}.

Under some assumptions, our results can be extended to a scheme $X$ defined over an algebraically closed field $k$ of characteristic zero (Proposition \ref{alg0}), or more generally over a  noetherian scheme $S$ (Proposition \ref{overS}).

By a theorem of Artin \cite[Exp XI,  Proposition  3.3]{SGA4}  any smooth scheme over an algebraically closed field  $k$ of  characteristic 0  can be covered by Artin neighborhoods. This was generalized  by Achinger \cite{ACH2},  by proving that any smooth scheme over a  field  of positive characteristic admits a cover by $K(\pi, 1)$ open subschemes. This gives  us the possibility  to make legitimately  assumptions on  one piece   of such a  cover. Therefore,  by using  purity theorems for \'etale cohomology  and  local to global comparison techniques,  we get our third main result:
\begin{theorem}{\rm (Theorem \ref {localbr})}   
Let  $X$ be a  smooth variety over an algebraically closed field $k$ of characteristic $p \geq 0$,   such that any pair of point $(x, y) \in X $  is contained in an affine open scheme. Suppose that there exists  an algebraic $K(\pi, 1)$ open subscheme    $   Y \subset X$   such that for every $z \in Z:= X - Y$,  the local ring $\shO_{X, z}$ has dimension $\geq 2$. Then $\Br(X)= \Br'(X)$ up to a $p$-component.
\end{theorem}
This partially extends a result of Grothendieck \cite[II,  Corollary 2.2]{GR1} for regular noetherian schemes of dimension less than 2 to varieties of arbitrary dimension.

Our main example of application is the case of an abelian variety $A$ over fields of characteristic zero. In Section \S \ref{S7} we apply our results along with  a general version of Riemann existence theorem for smooth algebraic groups  ( Lemma \ref{riemalg}),  to prove that $\Br(A) = \Br'(A)$ (Theorem \ref {abel}). This is an alternative proof to the ones proposed by Hoobler \cite{HOO1} and Berkovich \cite{BERK}.

In the last section \S \ref{S8}   we give a characterization of  smooth quasi-projective varieties $X$ with  $ \pi_2^\et(X, \bar{x})=0$ by  means  of  pro-universal covers (Proposition \ref {carpi}). This characterization  provides independently a sufficient condition under which the Brauer map is surjective. Namely, we prove the following:

\begin{theorem}{\rm (Proposition \ref {probr})}  
Let $X$  be a regular connected scheme of finite type over a field k of characteristic 0, and  $\hat f: \hat X \ra X$  the pro-universal cover of X, where $\hat{X}:= \varprojlim X_i$  is the limit of the  projective system of finite \'etale  covers of $X$. Suppose that $H_{\et}^2( \hat{X},  {\hat f}^{*}\shF)= 0$ for every locally constant constructible torsion \'etale sheaf $ \mathcal{F}$ on $X$. Then $\Br(X)= \Br'(X)$.
\end{theorem}
Here is the plan of the paper:  Section \S \ref{S1}  is  preliminaries on the construction of Brauer groups and Brauer map for  schemes. Section \S \ref{S2}  is devoted to the study of  analytic Brauer  groups with applications to the complex analytic space $X^{an}$. In section \S \ref{S3}  we review briefly the construction of the \'etale homotopy groups $\pi_n^\et(X,  \bar x)$ and algebraic $K(\pi,1)$ spaces, and we discuss  their resulting  consequences on the Grothendieck question. In sections \S \ref{S4} and \ref{S5}   we apply previous results  to solve the problem for some proper and smooth schemes. In section \S \ref{S6}  we discuss a local condition for the surjectivity of Brauer map for smooth varieties. Section \S \ref{S7}  is an application to abelian varieties. Pro-universal covers will be reviewed in section \S \ref{S8}  with its applications on our problem.
\begin{notation}
Throughout this paper we consider the following notations:
\begin{itemize}
 \item   A $\mathbf {variety }$ over a field $k$ is a  separated,  geometrically integral scheme of finite type over $k$. In particular a variety is quasi-compact and quasi-separated.
 \item $\mu_{n, X}$ the \'etale sheaf of $n$-th roots of unity on $X$.
 \item  $\GG_{m, X}$ the \'etale sheaf of multiplicative groups  on $X$.
 \item  If $Y \lra X$ is a morphism of schemes,  we denote by  $\mu_{n, Y}$ (resp. $\GG_{m, Y}$)   the pullback of $\mu_{n, X}$  (resp. of $\GG_{m, X}$)    to $Y$.
 \item  For an abelian group $A$,  $n$ an integer and $l$ a prime,  we denote by $A[n]$ the $n$-torsion subgroup  of $A$,  and by $A(l)$ the $l$-primary subgroup of $A$. The notations  ${}_n\! A$ and $A_n$  will stand for the kernel and the cokernel of the endomorphism $a \ra n.a$ of $A$.
 \item For a  field $k$, we denote by  $\bar k$  the separable closure of $k$. If $X$ is a scheme over $k$, and $k \subset K$ a field extension, $X_K= X \times_k K$ denotes the base change of $X$to $K$.
\end{itemize}
\end{notation}
%===========================================================
\section{Preliminaries}\label{S1}

Following Grothendieck \cite{GR1} and Milne  \cite[Chapter IV]{MIL},  we recall some elementary facts needed in  the sequel about Brauer groups of schemes.

 Let $X$  be a scheme.  An Azumaya algebra  $\shA$ on $X$  is a coherent $\shO_X$-algebra which is a locally free $\shO_X$-module of finite rank and satisfies one of the following equivalent conditions

\begin{itemize}
 \item [(i)]  For every point  $x \in X$,   $\shA_x$ is an Azumaya algebra over $\shO_{X, x}$.
\item [(ii)] For every point $x \in X$ the fiber $\shA_x \otimes k(x)$ is a central simple algebra over the residue field $k(x)$.
\item [(iii)] The natural morphism $\shA \otimes_{\shO_X} \shA^{op} \lra \shE nd_{\shO_X}(\shA)$ is an isomorphism.
\item [(iv)] There is a covering $(U_i \ra X)$ in the \'etale topology on $X$ such that for each $i$ there exists an $n_i$ such that $ \shA \otimes_{\shO_X} \shO_{U_i} \simeq M_{n_i}(\shO_{U_i})$.
\item [(v)] There is a covering $(U_i \ra X)$ in the flat topology on $X$ such that for each $i$ there exists an $n_i$ such that $ \shA \otimes_{\shO_X} \shO_{U_i} \simeq M_{n_i}(\shO_{U_i})$.
\end{itemize}
Two Azumaya algebras  $\shA_1$ and $\shA_2$ are called Morita equivalent if there exist locally free $\shO_X$-modules $\shM_1$ and $\shM_2$ of finite rank,  and an isomorphism
$$
\shA_1 \otimes_{\shO_X} \shE nd_{\shO_X}(\shM_1)  \simeq \shA_2 \otimes_{\shO_X} \shE nd_{\shO_X}(\shM_2) 
$$
The set of classes of Azumaya algebras on $X$ is a group called the Brauer group of $X$ and denoted by $\Br(X)$. The group law is given by tensor product,  the  inverse of a class $[\shA]$ is the class of its opposite algebra $[\shA^{op}]$,  and the unit element has the form $\shE nd_{\shO_X}(E)$,  where $E$ is a locally free  $\shO_X$-module.

Fix an integer $n \geq 2$, and consider the following exact sequence of \'etale sheaves on $X$  (cf.  \cite[Chapter IV, Corollary 2.4]{MIL})
$$ 1 \longrightarrow  \mathbb{G}_{m,X}    \longrightarrow  \GL_n(\mathcal{O}_X)  \lra  \PGL_n(\mathcal{O}_X)   \longrightarrow 1$$ 
Non abelian cohomology yields an exact sequence of \v Cech cohomology groups
$$...  \longrightarrow   \check{H}_{\et}^1(X, \mathbb{G}_{m,X} )   \longrightarrow    \check{H}_{\et}^1(X,  \GL_n(\mathcal{O}_X)  \lra  \check{H}_{\et}^1(X,   \PGL_n(\mathcal{O}_X)    \overset{\delta_n}\longrightarrow  \check{H}_{\et}^2(X, \mathbb{G}_{m,X})$$ 

The following is a fundamental result in the theory of Brauer groups of schemes.
\begin{theorem} \label{groth}
Let  $X$  be a scheme,  then we have the  following statements
\begin{itemize}
\item [(i)]  The set of classes of Azumaya algebras  of  rank $n^2$ is isomorphic  to the cohomology group $\check{H}_{\et}^1(X,   \PGL_n(\mathcal{O}_X)$.
\item [(ii)]  The  maps ${\delta_n}$ induce a  group homomorphism $\delta': \Br(X) \ra H_{\et}^2(X, \GG_{m,X})$.
\item [(iii)]   This homomorphism  $\delta': \Br(X) \ra H_{\et}^2(X, \GG_{m,X})$ is injective.
\item [(iv)]    $Im({\delta_n}) \subseteq  H_{\et}^2(X, \GG_{m,X})[n]$.
\end{itemize}
\end{theorem}
\proof
This is the original Grothendieck statement   \cite[I,  Proposition 1.4]{GR1}. Milne gave in  \cite[Chapter IV,  Theorem 2.5 and Proposition 2.7 ]{MIL}  a proof for  \v Cech cohomology and another general proof for \'etale cohomology by means of  gerbes theory.
\qed

\medskip
The group $H_{\et}^2(X, \GG_{m,X})$ is called the cohomological Brauer group,  or Brauer-Grothendieck group. As observed by Grothendieck,  the map $\delta': \Br(X) \ra H_{\et}^2(X, \GG_{m,X})$ is not bijective in general. Indeed,  for quasi-compact schemes $\Br(X)$ is always torsion (\cite[I,  Section 2]{GR1}),  while there exists a normal surface $S$ such that $H_{\et}^2(S, \GG_{m,S})$  is not torsion (\cite[II,  1.11.b]{GR1}).

Let $X$ be a quasi-compact scheme. Denote by $\Br'(X) := H_{\et}^2(X, \GG_{m,X})_{\rm tor}$ the torsion part of the cohomology group  $H_{\et}^2(X, \GG_{m,X})$,  and consider the map
$$\delta: \Br(X) \lra \Br'(X)$$
called the Brauer map. Grothendieck asked the the following question: 

\medskip
$\mathbf {QUESTION: }$ Is $\delta: \Br(X) \lra \Br'(X)$ surjective for quasi-compact schemes ?
\medskip

To answer this question, almost all our results in this paper will be based on the two following fundamental lemmas. The first one requires an affine  topological condition.
\begin{lemma}\cite[Proposition 3]{HOO2} 
\mylabel{etale}
Let $X$ be a scheme such that any finite set of points is contained in affine open scheme,  and  $\alpha \in H_{\et}^2(X, \GG_{m, X})$. If there exists a finite \'etale cover $f: Y \rightarrow X$ such that $f^{*}(\alpha)$=0 in $ H_{\et}^2(Y, \GG_{m, Y})$, then $\Br(X)= \Br'(X)$.
\end{lemma}
\begin{remark}\label{h2cech}
The condition proposed  above was in fact used because it implies that \v Cech and \'etale cohomology coincide for $X$ (cf. \cite[Section 4]{ARTN}). 
However,   since we are just  dealing  with  2-cohomology classes,  we can use the refined version of Schr\"oer \cite[Corollary 2.2]{SCH2} which states that  if any   pair $(x, y) \in X$ is contained in affine open scheme,  then the  2-\v Cech cohomology and 2-\'etale cohomology agree. For example, such a condition holds for open subschemes of  toric varieties (cf. \cite[Corollary 2.3]{SCH2}). 
\end{remark}

For a general scheme   for which  condition in Lemma \ref{etale} need not  be satisfied,  we have the following criterion  which involves Galois covers.

\begin{lemma}\cite[Corollary 1]{BERK} 
\mylabel{galois}
 Let  $X$ be  a regular scheme. If for every $\alpha \in H_{\et}^2(X, \GG_{m, X})$ there exists an \'etale Galois cover $g: Y \rightarrow X$ such that $g^{*}(\alpha)$=0 in $ H_{\et}^2(Y, \GG_{m, Y})$, then  $\Br(X)= \Br'(X)$.
\end{lemma}
\begin{remark}
If $\beta \in \Br'(X)$,  then by Theorem \ref{groth}.(iv)  there is an integer $n$ with $n.\beta$= 0,  and a class $\alpha \in  H_{\et}^2(X, \mu_{n, X})$ mapping to $\beta$. Therefore,   covers in the above lemmas  can be chosen sufficiently  to trivialize   classes  of  $ H_{\et}^2(X, \mu_{n, X})$. This comes with very nice consequences, since the \'etale sheaf $\mu_{n, X}$  belongs to the category of  locally constant constructible torsion \'etale sheaves, which will play a crucial rule in this paper.

\end{remark}

\section{Case of  schemes over $\CC$}\label{S2}

\mylabel{analytic brauer}
In this section,  we give some elementary results on the analytic Brauer group,  which are  closely related to the Brauer group of schemes over complex numbers.

Let $(X, \shO_X)$ be a complex analytic space,  where $\shO_X$ is the sheaf of holomorphic functions. An Azumaya algebra on  $X$ is an associative
(non-commutative)  $\shO_X$-algebra $\shA$ which is locally (in the analytic topology) isomorphic to a matrix algebra $M_n(\shO_X)$ for some $n > 0$. Working with cohomology of sheaves, all facts  in the previous section could  be applied to define  the analytic Brauer group $\Br(X)$ of $X$,  and hence we get a well defined injective Brauer map $\delta: \Br(X) \ra  \Br'(X):=H^2(X, \GG_{m, X})_{\rm tor}$ (see \cite{HS} for more details). Equivalently, one can define $\Br(X)$ as the set of equivalence classes of  principal $\PGL_n$-bundles via the boundary  maps $\delta_n:  H^1(X,   \PGL_n(\mathcal{O}_X))   \ra  H^2(X, \mathbb{G}_{m,X})$ (cf. \cite{SCH3}). In particular, when $X$ is  a topological $K(\pi,1)$ space (see definition bellow), the group $\Br(X)$ can be defined  in terms of projective representations $\rho: \Gamma \ra \PGL_n(\CC)$ of the topological fundamental group $\Gamma = \pi_1(X)$. This fact was used by Elencwajg and Narasimhan  \cite{EN} to prove that $\Br(X) = \Br'(X)$ for complex torus.

Schr\"oer proved (see theorem below) that the surjectivity of the Brauer map  $\delta: \Br(X) \ra \Br'(X)$   for complex analytic spaces depends only on the homotopy type of their underlying topological space. He  used the following notion of  $\mathbf {good \; groups}$  introduced by Serre in \cite{SER}: Let $G$ be a group endowed with the discrete topology,  and $\hat G = \varprojlim G/N$ its profinite completion, where the limit runs over all normal subgroups $N \subset G$. By construction, the group $\hat G$
 carries the inverse limit topology. Let $M$ be a finite discrete $G$-module,  that is a $G$-module which is  finite as a set. The action of $G$ induces a natural action of $\hat G$ on $M$. We say that $G$ is good  or of type $\shA_n$ \cite[Chapter I,  \S 2.6, b.2]{SER}  if the natural morphism of cohomology groups
$$
H^n(\hat G,  M) \lra H^n(G, M)
$$
induced by the natural morphism $G \lra \hat G$ is an isomorphism for all $n \geq 0$. The following types of groups are examples of  good groups:
\begin{itemize}
\item Free groups,  finite groups.
\item Almost free groups,  almost polysyclic groups (see \cite{SCH3}).
\item Bianchi groups $\PSL(2,\shO_d)$, where $\shO_d$ is the ring of integers in an imaginary quadratic number filed $\QQ(\sqrt{-d})$  (\cite{GJZ}).
\item Right-angled Artin groups (\cite{KL}).

\end{itemize}

Note that   Lorensen \cite{KL} developed a construction by means of Mayer-Vietoris sequence to prove that  free products with amalgamation and HNN extensions of some classes of good groups, are in fact good groups. This construction was also employed by Grunewald-Zapirain-Zalesskii in \cite{GJZ} along with some equivalent properties of goodness to provide many examples of good groups including Bianchi groups (eg. $\shF$-groups, Limit groups,...).

\begin{theorem}
\label{schr}
\cite[Theorem 4.1]{SCH3}  Let  $X$ be   a complex analytic space.  Suppose that the topological fundamental group $\pi_1(X)$   is good,  and that  the  subgroup of $\pi_1(X)$-invariants inside the Pontryagin dual is trivial, i.e. $\Hom(\pi_2(X),  \QQ / \ZZ) ^{\pi_1(X)}=0$. Then  $\Br(X)= \Br'(X)$.
\end{theorem}

Let $X$ be a scheme of finite type over  over $\CC$. There is an associated analytic space $X^{an}$ whose underlying topological space is $X(\CC)$ the space of $\CC$-rational points of $X$. The following is  a first elementary result describing the link between Brauer maps for $X$ and $X^{an}$.

\begin{proposition}\label{desbr}
Let  $X$ be a    scheme of finite type over $\CC$.  Suppose that  $X^{an}$ is compact. Then  $\Br(X)= \Br'(X)$    if only if    $\Br(X^{an})= \Br'(X^{an})$. 
\end{proposition}
\proof
Consider the following commutative diagram 
$$
\begin{CD}
\Br(X) @>>> \Br(X^{an})\\
@VVV @VVV\\
 H_{\et}^2(X, \GG_{m, X}) @>>> H^2(X^{an}, \GG_{m, X^{an}})
\end{CD}
$$
The upper map is an isomorphism according to \cite[Proposition 1.4]{SCH3}. By   comparing  cohomology exact sequences induced by the Kummer exact  sequence for both $X$ and $X^{an}$, and  using the fact that   $ H_{\et}^2(X, \mu_{n, X}) \simeq  H^2(X^{an}, \mu_{n, X^{an}})$  by Artin  comparison theorem \cite[Exp. XVI, Theorem 4.1]{SGA4}  we conclude that the lower map is an isomorphism (see \cite[Proposition 1.3]{SCH3}). Hence the assertion.
\qed

\medskip

\begin{example}
Let $X$ be an  algebraic  K3 surface over $\CC$,  that is a complete non-singular variety of dimension two over $\CC$  such that $\Omega^2_{X/\CC} \simeq  \mathcal{O}_X$ and $H_{\zar}^1(X,  \mathcal{O}_X) = 0$. Its associated analytic space $Y=X^{an}$  is a complex K3 surface,  i.e. a  compact connected complex manifold of dimension two such that  $\Omega^2_{X} \simeq  \mathcal{O}_X$ and $H^1(X,  \mathcal{O}_X) = 0$.  According to  Huybrechts and  Schr\"oer result for analytic  K3 surfaces \cite[Theorem 1.1]{HS}  we have $\Br(Y) = \Br'(Y)$,  hence  Proposition \ref{desbr} asserts that $\Br(X) = \Br'(X)$.
\end{example}

In order to apply Theorem \ref{schr} to the space $X^{an}$ we need the two following notions of  $\mathbf {Artin \;  neighborhoods}$  and   $\mathbf {topological   \;  K(\pi,1)  \;  spaces } $.
 
Following  Artin \cite[Exp XI,  Section 3]{SGA4}
a morphism of schemes $f : X \ra S $ is called an elementary fibration
if there exists a  factorization   
\begin{center}

\begin{tikzpicture}
  \matrix (m) [matrix of math nodes, row sep=4em, column sep=4em, minimum width=2em]
 {
 X & \bar X  & Y \\
      & S & \\};
  \path[-stealth]
    (m-1-1) edge node [left] {$f$} (m-2-2)
            edge  node [above] {$j$} (m-1-2)
    (m-1-2) edge node [right] {$\bar f$} (m-2-2)
  (m-1-3)    edge node [above] {$i$}         (m-1-2)  
    (m-1-3) edge node [right] {$g$} (m-2-2);

\end{tikzpicture}
\end{center}
such
that
\begin{itemize}
\item [(i)]  $j$  is an open immersion and $X$ is fiberwise dense in $\bar X$.

\item [(ii)] $\bar f $  is a smooth and projective morphism whose geometric fibers are nonempty
irreducible curves.

\item [(iii)] The reduced closed subscheme $Y= \bar X  \backslash  X$  is finite and \'etale over $S$.
\end{itemize}
Let $k$ be a field. An Artin neighborhood (or a good neighborhood ) over  $\Spec(k)$ is a  scheme $X$  over $k$ such that there exists a sequence of $X$-schemes
$$X=X_n,..., X_0=\Spec(k)$$
with elementary fibrations $f_i: X_i \ra X_{i-1}$,  $i=1,..., n$.

Let $G$ be a group and $n$ a positive integer. A connected topological space $X$ is called an Eilenberg-MacLane space of type $K(G, n)$  if it has $n$-th homotopy group $\pi_n(X)$  isomorphic to $G$ and all other homotopy groups  trivial. In particular $X$ is called  a topological $ K(\pi,  1)$ space if it is weakly homotopy equivalent to the classifying space  $B\pi_1(X)$,  that is $\pi_n(X)=0$ for all $n \geq 2$. An equivalent definition of topological $ K(\pi,  1)$ spaces  in terms of cohomology is given as follows: Let  $x \in X$, there is a fully faithful functor 
$$ \rho^{*}: \pi_1(X,x){\rm-Mod}  \lra Sch(X)$$ from the category of $\pi_1(X,x)$-modules to the category of sheaves on $X$, whose essential image is the subcategory of locally constant sheaves on $X$. It associates to any  $\pi_1(X,x)$-module $M$ a locally constant sheaf $ \rho^{*}(M)$, with $(\rho^{*}(M))_x = M$ and $\Gamma(X,\rho^{*}(M)) = M^{\pi_1(X,x)}$. Therefore, the formalism of  universal $\delta$-functors gives rise to natural morphisms of cohomology groups $$\rho^{q}: H^q(\pi_1(X,x),M) \lra H^q(X, \rho^{*}(M))$$
The space $X$ is a topological $ K(\pi,  1)$ space if only if the morphisms $\rho^{q}$ are isomorphisms for all $q \geq 0$.
\begin{lemma}\label{khtp}
Let $X$ be a  connected scheme of finite type over $\CC$. If  X is an Artin neighborhood over $\Spec(\CC)$,  then
\begin{itemize}
\item [(i)]  $X(\CC)$ is a topological $K(\pi, 1)$ space.
\item [(ii)]   $\pi_1(X(\CC))$  is a good group.
\end{itemize}
\end{lemma}
\proof 
This is proven by Serre  \cite[Exp XI,  Variant  4.6]{SGA4} as a variant of the proof of Artin comparison theorem \cite[Exp. XI, Theorem 4.4]{SGA4}.  The result follows from the fact that if $X \ra S$ is an elementary fibration,  then $X(\CC) \ra S(\CC)$ is a locally trivial topological fiber bundle whose fiber $F$ is a topological  $K(\pi, 1)$ space and its fundamental group $ \pi_1(F)$ is free of finite type. The exact sequence of homotopy groups
$$...\ra  \pi_n(F)   \ra \pi_n(X(\CC)) \ra  \pi_n(S(\CC))  \ra    \pi_{n-1}(F)  \ra...$$
implies that $X(\CC)$ is a topological $K(\pi, 1)$ space and  $\pi_1(X(\CC))$  is a succession of extensions of free group of finite type,  whence by \cite[Chapter I,  \S 2.6 2.d]{SER} it is a good group.
\qed

\medskip
\begin{proposition}
Let $X$ be a  connected scheme of finite type over $\CC$. If  X is an Artin neighborhood over $\Spec(\CC)$, then $\Br(X^{an})= \Br'(X^{an})$.
\end{proposition}
\proof 

Since $X(\CC)$ is a topological $K(\pi, 1)$ space,  then in particular   $\pi_2(X(\CC))=0$. Hence the assertion follows from Lemma \ref{khtp} and Theorem \ref{schr}.
\qed

\medskip
\begin{corollary}
Let $X$ be a connected scheme  of finite type over $\CC$. If  X is an  Artin neighborhood over $\Spec(\CC)$  and  $X^{an}$ is compact, then   $\Br(X)= \Br'(X)$.
\end{corollary}
\begin{corollary}\label{propb}
Let $X$ be a  proper connected  scheme of finite type over $\CC$. If  X is an  Artin neighborhood over $\Spec(\CC)$,  then   $\Br(X)= \Br'(X)$.
\end{corollary}
\proof
Since $X$ is a proper scheme of finite type over $\CC$, then $X^{an}$ is compact.
\qed

\medskip
\begin{corollary}
Let $X$ be a smooth  proper connected scheme of finite type over $\CC$. There is an open $ U \subseteq X$ such that   $\Br(U)= \Br'(U)$.
\end{corollary}
\proof
By  Artin theorem \cite[Exp XI,  Proposition  3.3]{SGA4},  $X$ admits  a cover by Artin neighborhoods.
\qed

\medskip
The purpose of the next sections is to apply these results  to study the case of  proper and smooth schemes over subfields of $\CC$. This involves the algebraic version of $K(\pi,1)$ spaces which is closely related to the \'etale homotopy type.

\section{\'Etale homotopy type and  $K(\pi, 1)$ spaces}\label{S3}

We begin this section by a brief summary of  Artin and Mazur  construction of  \'etale homotopy type and \'etale homotopy groups. The standard reference for this is \cite{AM}.

Let $X$ be a locally noetherian connected scheme,  and denote by $Cov(X)$ the category of \'etale covers of  $X$,  and by $Hyp(X)$ the category of \'etale hypercoverings of $X$. Any object  $\mathcal U_ \bullet$ of  $Hyp(X)$  is  a simplicial object of $Cov(X)$ \cite[Definition 8.4]{AM}. Since every object $Y  \rightarrow X$ in $Cov(X)$ is a disjoint union of connected schemes,  we can consider the functor $\pi_0: Hyp(X) \rightarrow Sets$,  where $\pi_0(Y)$ is the set of connected components of $Y$. It extends to  a functor $\pi_0: SHyp(X) \rightarrow SSets$ from the category of simplicial \'etale hypercoverings of $X$  to the category of simplicial sets,  and by taking the quotient with simplicial homotopy we get a functor
$
 \{\pi_0(-)\}:   Ho(SHyp(X)) \rightarrow Ho(SSets)
$
of  homotopy categories. Since $ Ho(SHyp(X)) $ is cofiltering \cite[Corollary 8.13.(i)]{AM},  then one can define  the \'etale homotopy type $Et$  as an object in pro-Top
$$
Et:  Ho(SHyp(X)) \lra \rm Top  
$$
in the following sense: Take an hypercovering $\mathcal U_ \bullet$ of $X$,  and put $ \pi_0(X):=  \{\pi_0(\mathcal U_ \bullet)\} $. Then one defines $Et(X):= |\pi_0(X)|$,  where $|S|$ is the topological realization of the simplicial set $S$. Such a topological space can be given the structure of a CW-complex,  hence $Et(X)$ is an object in pro-$\mathcal H$,  the pro-category of the homotopy category of CW-complexes.

For any abelian group $A$ we have a canonical isomorphism \cite[Corollary 9.3]{AM}
$$
H^n( Et(X),  A)  = H_{\et}^n(X, A)
$$
A given   geometric point $\bar x$  of $X$ defines a point $ \bar x_{\et}$ on $Et(X)$,  hence one can define the \'etale homotopy groups for all $n \geq 0$:
$$
\pi_n^\et(X, \bar{x}):= \pi_n(Et(X),  \bar x_{\et})
$$
In particular by \cite[Corollary 10.7]{AM} $\pi_1^\et(X, \bar{x})$ is the usual \'etale fundamental group defined by Grothendieck in \cite{SGA1}.

\begin{remark}
Let $\shF$ be a locally constant constructible $n$-torsion \'etale sheaf on $X$ for some  integer $n$,  it can be written as follows
$$\shF  = \bigoplus_{i=1}^r (\ZZ/n^{p_i}\ZZ)^{m_i}$$
 where $p_i$ and $m_i$ are positive integers. Hence we have a natural identification for all $q \geq 0$
 $$H^q( Et(X),  \shF ) : =   \bigoplus_{i=1}^r H^q( Et(X),  (\ZZ/n^{p_i}\ZZ)^{m_i}) =   \bigoplus_{i=1}^r H_{\et}^q(X,  (\ZZ/n^{p_i}\ZZ)^{m_i})$$
\end{remark}
\begin{lemma}
\mylabel{homotop}

let  $f:  (Y,  \bar y) \lra (X,  \bar x)  $ be a finite \'etale surjective morphism of pointed connected schemes,  then
$$
\pi_n^\et(Y, \bar y)    \simeq   \pi_n^\et(X, \bar x) 
$$
for all $n \geq 2$.
\end{lemma}

\proof
For   smooth connected  quasi-projective varieties  over an algebraically closed field $k$,  this is \cite[Proposition 4.1]{PAL}.
For arbitrary connected schemes,   the assertion   follows from \cite[Lemma 2.1]{SCHM}.
\qed

\medskip
The following result is a generalization of Theorem \ref{schr}. Since $ \pi_1^\et(X, \bar{x})$ is always profinite, 
we  use properties of continuous cohomology of profinite groups,  and hence we can omit   the goodness assumption. Furthermore,   Lemma  \ref{homotop} will serve  to get  the desired  \'etale Galois cover which kills cohomological Brauer classes.
\begin{theorem}
\mylabel{rcnb}
Let $X$ be a regular  connected  scheme of finite type over a field k of characteristic 0, with a geometric base point $\bar x \ra X$,   such that $ \pi_2^\et(X, \bar{x})=0$. Then $\Br(X) = \Br'(X)$.
\end{theorem}

\proof
Let  $p: (Et(X)^\thicksim, \tilde{x}_\et) \ra (Et(X),\bar x_{\et})$   be the  universal cover of the \'etale homotopy type $Et(X)$. For any locally constant constructible torsion \'etale  sheaf $\mathcal{F}$ on $X$, we have a spectral sequence
$$
E^{p,q}_2=H^p(\pi_1^\et(X, \bar x), H^q(Et(X)^\thicksim, p^{*}\shF)) \Ra H_{\et}^{p+q}(X, \shF)
$$
This is in fact a Grothendieck spectral sequence associated to the functor
$$
\Gamma(Et(X)^\thicksim,  p^{*}(-)): Sch(Et(X))  \lra \pi_1^\et(X, \bar x)\rm{-Mod} 
$$
from the category of sheaves  on $Et(X)$ to the category of $\pi_1^\et(X, \bar x)$-modules,  and the functor
$$
(-)^{\pi_1^\et(X, \bar x)}: \pi_1^\et(X, \bar x)\rm{-Mod} \lra Ab
$$
from the category of $\pi_1^\et(X, \bar x)$-modules to the category of abelian groups.
Therefore,  we get an exact sequence of low-degree terms
\begin{align}\label{h12}
0 \ra  H^1(\pi_1^\et(X, \bar{x}),   H^0( Et(X)^\thicksim,  p^{*}\mathcal{F}) )      \ra   H_{\et}^1(X, \mathcal{F})  & \ra   H^0(\pi_1^\et(X, \bar{x}),   H^1(Et(X)^\thicksim,  p^{*}\mathcal{F}))   \\
 \ra  H^2(\pi_1^\et(X, \bar{x}),   H^0( Et(X)^\thicksim,  p^{*}\mathcal{F}) )     \ra  H_{\et}^2(X, \mathcal{F}) & \ra   H^0(\pi_1^\et(X, \bar{x}),   H^2( Et(X)^\thicksim,  p^{*}\mathcal{F})) \notag
\end{align}
We have $H^0( Et(X)^\thicksim,  p^{*}\mathcal{F}) =  \mathcal{F}_{\bar x}$ and $H^1(Et(X)^\thicksim,  p^{*}\mathcal{F})  =0$.
By the topological Hurewicz theorem we  get an isomorphism 
$$
 H_2(  Et(X)^\thicksim,  \ZZ)      \simeq          \pi_2 (Et(X)^\thicksim, \tilde{x}_\et)    \simeq   \pi_2 ( Et(X), \bar{x}_\et) = \pi_2^\et(X, \bar{x})  
$$
And  by the universal coefficient theorem we have 
$$
 H^2( Et(X)^\thicksim,  p^{*}\mathcal{F})  \simeq  \Hom(  H_2(  Et(X)^\thicksim,  \ZZ),  \mathcal{F}_{\bar x})
$$
Hence  we get a short exact sequence
$$
0\lra  H^2(\pi_1^\et(X, \bar{x}),   \mathcal{F}_{\bar x})    \lra  H_{\et}^2(X, \mathcal{F})   \lra \Hom( \pi_2^\et(X, \bar{x}),  \mathcal{F}_{\bar x})^{\pi_1^\et(X, \bar{x})} 
$$
By  assumption on $\pi_2^\et(X, \bar{x})$,   and  in light of  \cite[Chapter I, \S 2.2,  Corollary 1]{SER}  we have an isomorphism
\begin{align}
H_{\et}^2(X,  \mathcal{F}) & \simeq   H^2(\pi_1^\et(X, \bar{x}), \mathcal{F}_{\bar x}) \\
& \simeq \varinjlim H^2(\pi_1^\et(X, \bar{x})/ N,  \mathcal{F}_{\bar x}^N) \notag
\end{align}
where the limit runs over all normal open subgroups $N$ of $\pi_1^\et(X, \bar{x})$,  and 
 $\mathcal{F}_{\bar x}^N$ is the submodule of $N$-invariant elements.
Next,  take $\mathcal{F }=\mu_{n, X}$ for some integer $n$, and choose a class $\beta \in H_{\et}^2(X,  \mu_{n, X})$,  it belongs  to a group $H^2(\pi_1^\et(X, \bar{x})/ N,   (\mu_{n, X})^N_{\bar x})$ for some open normal subgroup $N$. Further, $N$ is of  finite index since it is an open normal subgroup of a profinite group,  thus $G:= \pi_1^\et(X, \bar{x})/ N$ is a finite quotient of  $\pi_1^\et(X, \bar{x})$. Therefore, the fundamental Galois correspondence implies that there exists a pointed \'etale Galois cover $f: (Y,  \bar y)  \rightarrow (X,  \bar x)$ with Galois group $G$ and $\pi_1^\et(Y, \bar{y}) = N $. On the other hand, Lemma \ref{homotop} asserts that $\pi_2^\et(Y, \bar{y}) = 0$,  hence we get by the same argument an isomorphism
$$
H_{\et}^2(Y,  \mu_{n, Y})  \simeq   H^2(N,   (\mu_{n, Y})_{\bar y} ) 
$$
Since the map 
$$   H^2(\pi_1^\et(X, \bar{x})/ N,  (\mu_{n, X})^N_{\bar x})   \lra      H^2(N,   (\mu_{n, X})_{\bar x} ) =     H^2(N,  (\mu_{n, Y})_{\bar y} )   $$
 is zero,   we conclude that  the image of $\beta$ under the map 
 $$f^{*}:  H_{\et}^2(X,  \mu_{n, X} ) \lra  H_{\et}^2(Y,  \mu_{n, Y})$$
is zero,  hence it follows from  Lemma \ref{galois} that $\Br(X) = \Br'(X)$.
\qed

\medskip
Following Achinger \cite{ACH1} and \cite{ACH2}, we consider the notion of algebraic $K(\pi,  1)$ spaces, which is defined only for   coherent schemes  that have finitely many components. By coherent we mean quasi-compact and quasi-separated scheme. In our context, we consider  connected noetherian schemes which belongs to this class. Further,  we adopt the second definition introduced in \cite{ACH2} which does not require sheaves of order invertible on $X$. Note that algebraic $K(\pi,  1)$ spaces are defined in \cite[2.3]{SCHM} in terms of \'etale homotopy groups. The two definitions are equivalent in the case of geometrically unibranch schemes (Proposition \ref{kphom}).

Let $X$ be a noetherian scheme,  and denote by $X_{\et}$ (resp. $ X_{\fet}$) the \'etale site (resp. the finite \'etale site) of $X$. The forgetful functor from the category of finite \'etale  covers of $X$ to the category of \'etale covers induces   a natural morphism of sites
$$
\rho: X_{\et}  \lra     X_{\fet}
$$
If $X$ is connected,  then for a given geometric point $\bar{x} \rightarrow X$,  the site  $ X_{\fet}$  is equivalent to the classifying site $B\pi_1^\et(X, \bar{x})$ whose underlying category  is the category of continuous $\pi_1^\et(X, \bar{x})$-sets. For every locally constant torsion \'etale sheaf $\shF$ on $X$ and $q \geq 0$,  we have then a natural morphism
$$
\rho^q: H^q( \pi_1^\et(X, \bar{x}), \mathcal{F}_{\bar x} )  \simeq H^q_{\fet}(X,  \rho_{*}\mathcal{F})       \lra   H_{\et}^q(X,  \mathcal{F})
$$
 \begin{definition}\label{defkp}(\cite{ACH1}, \cite{ACH2})
A pointed connected noetherian  scheme  $(X, \bar x)$ is an algebraic $K(\pi,  1)$ space if for every  locally constant constructible torsion \'etale sheaf $\mathcal{F}$
on $X$, the natural morphisms
$$
\rho^q:  H^q( \pi_1^\et(X, \bar{x}), \mathcal{F}_{\bar x} )       \lra   H_{\et}^q(X,  \mathcal{F})
$$
are isomorphisms for all  $q \geq 0$.

 \end{definition}

\begin{example}\label{expkp} The following schemes are examples of algebraic  $K(\pi,  1)$ spaces:
\begin{itemize}
\item The spectrum of a field $\Spec(k)$.
\item Smooth connected curves $C$ of genius $g > 0$ \cite{SCHM}.
\item Abelian varieties ( see proof of Theorem \ref{abel}).
\item Finite product of geometrically connected and geometrically unibranch $K(\pi, 1)$ varieties over a field $k$ of characteristic zero \cite{SCHM}.
\item Connected affine $\FF_p$-schemes \cite{ACH2}.
\end{itemize}

\end{example}

Recall that a scheme $X$  is geometrically unibranch if  for every $x \in X$ the local ring $\shO_{X, x}$ is geometrically unibranch (\cite[6.15.1]{EGA}). In particular any normal scheme is  geometrically unibranch (\cite[Proposition 6.15.6]{EGA}).
\begin{proposition}\cite[Proposition 4.4]{ACH2} 
\label{kphom}
 Let $(X, \bar x)$  be a pointed  noetherian,  geometrically unibranch  connected  scheme.  Then $X$  is an algebraic $K(\pi,  1)$ space  if only if  $\pi_n^\et(X, \bar{x}) =0$ for all $n \geq 2$.
\end{proposition}
\begin{proposition}\cite[Proposition 3.2]{ACH1} 
\label{kpcover}
Let $X$ be a connected noetherian scheme. The following statements  are equivalent 
\begin{itemize}
\item [(i)] $X$ is an algebraic $K(\pi,  1)$ space.
\item [(ii)] For every locally constant constructible torsion \'etale sheaf  $\mathcal{F}$ on $X$,  and every class $\beta \in  H_{\et}^q(X,  \mathcal{F})$ with $q \geq 1$,  there exists a finite \'etale   cover $f: Y \rightarrow X$ such that $f^{*}(\beta) = 0$ in  $H_{\et}^q(Y,  f^*\shF)$.
\end{itemize}
\end{proposition}
The following lemma will be needed in the last section.
\begin{lemma}\label{h1}
Let $(X,  \bar x)$  be a pointed connected noetherian   scheme. Then  
\begin{itemize}
\item [(a)] For any locally constant  constructible torsion \'etale sheaf $\shF$ on $X$ we have 
$$H_{\et}^1(X,  \mathcal{F}) \simeq   H^1(\pi_1^\et(X, \bar x),  \mathcal{F}_{\bar x})$$

\item [(b)] If $Y \ra X$ is a finite \'etale cover, then $Y$ is an algebraic $K(\pi,1)$ space if only if $X$ is.

\end{itemize}

\end{lemma}

\proof (a): This  follows  from the exact  sequence  (\ref{h12}) in the proof of  Theorem \ref{rcnb}
$$0 \ra  H^1(\pi_1^\et(X, \bar{x}),   H^0( Et(X)^\thicksim,  p^*\mathcal{F}))       \ra   H_{\et}^1(X, \mathcal{F})   \ra   H^0(\pi_1^\et(X, \bar{x}),   H^1(Et(X)^\thicksim,  p^*\mathcal{F}))$$
and the fact that  $H^0( Et(X)^\thicksim,  p^*\mathcal{F}) =  \mathcal{F}_{\bar x}$ and $H^1(Et(X)^\thicksim,  p^*\mathcal{F})  =0$.
For an Alternative proof by means of torsor interpretation of  $H_{\et}^1(X, \shF)$ see \cite[Lemma 4.3]{ACH2}.

(b): This is \cite[Proposition 3.2.(b)]{ACH1}. Alternatively, since    $Y \ra X$ is a finite \'etale surjective morphism. Then  $X$ is normal if only if $Y$ is. Therefore, for the normal case, the statement can be deduced  from Lemma \ref{homotop} and Proposition \ref{kphom}.
\qed

\medskip

\begin{remark}\label{kpdep}

It should be pointed out that assertions in Proposition \ref{kphom} and Proposition \ref{kpcover}  depend on the integer $q$,  that is in particular for a pointed connected noetherian geometrically unibranch  scheme $(X,  \bar x)$ and $q=2$, the following statements are equivalent

\begin{itemize} 
\item [(i)] $\pi_2^\et(X, \bar{x}) =0$.

\item [(ii)] $H_{\et}^2(X,  \mathcal{F}) \simeq   H^2(\pi_1^\et(X, \bar x),  \mathcal{F}_{\bar x})$ for every locally constant constructible torsion \'etale sheaf $\mathcal{F}$ on $X$.

\item [(iii)] For every locally constant constructible torsion \'etale  sheaf  $\mathcal{F}$ on $X$,  and every class $\beta \in  H_{\et}^2(X,  \mathcal{F})$,  there exists a finite \'etale   cover $f: Y \rightarrow X$ such that $f^{*}(\beta) = 0$ in  $H_{\et}^2(Y,  f^{*}\shF)$.
\end{itemize} 
This is largely  sufficient for our purposes in this paper.

\end{remark}

\begin{proposition}\label{kp1finit}

Let $X$ be a connected scheme over a field k of characteristic 0,  such that any pair of points $(x, y) \in X$ is contained in an affine open scheme.  If $X$ is a $K(\pi,  1)$ space. Then  $\Br(X)= \Br'(X)$. 

\end{proposition}

\proof

Let  $\alpha \in  H_{\et}^2(X, \mu_{n, X})$ for some integer $n$,  by Proposition \ref{kpcover} there exists a finite \'etale cover $f: Y \rightarrow X$ such that $f^{*}(\alpha)=0$,  hence by Lemma \ref{etale} $\Br(X)= \Br'(X)$. 
\qed

\medskip

\begin{proposition}

Let $X$ be a  regular   connected scheme over a field k of characteristic 0.  If $X$ is a $K(\pi,  1)$ space. Then  $\Br(X)= \Br'(X)$. 

\end{proposition}

\proof Every regular scheme is normal, and  hence geometrically unibranch.  Thus the assertion follows from Theorem \ref{rcnb}  and Proposition \ref{kphom}.
\qed

\medskip

It is proven in \cite{DJNG} that $\Br(X)= \Br'(X)$ when $X$ is a scheme with ample invertible sheaf. This   holds when $X$ is in particular a regular quasi-projective geometrically irreducible variety over a field $k$. On another hand,  it is pointed out in \cite[Section 4]{ARTN}  that for a  quasi-projective variety  $X$ over a field $k$,  any finite set of points of $X$ is contained in affine open scheme, thus one can deduce the following immediate corollary.

\begin{corollary}

Let $X$ be a connected,  quasi-projective variety over a field k of characteristic 0.  If $X$ is a $K(\pi,  1)$ space, then  $\Br(X)= \Br'(X)$.

\end{corollary}

\section{Proper case: Descent of Brauer maps}\label{S4}

Proper schemes over algebraically closed fields have nice properties such that the stability of the \'etale fundamental group and \'etale cohomology groups after base changing to another algebraically closed field. The cohomological Brauer group behaves in the same way in this case.
\begin{proposition}\label{propbr}

 Let $k \subset K$  be an extension of algebraically closed fields of characteristic 0,  and let $X$ be a  proper,  geometrically connected  scheme of finite type over the field $k$. Then  $\Br'(X)= \Br'(X_{K})$.

\end{proposition}

\proof

Fix an integer $n$, and consider the Kummer exact sequence 
$$
1 \lra \mu_{n,X}  \lra \GG_{m, X}  \overset{x \ra x^n}\lra \GG_{m, X} \ra 1
$$
The corresponding exact sequence of  cohomology yields a  short exact sequence
$$
0 \lra \Pic(X)_n   \lra H_{\et}^2(X, \mu_{n,X})   \lra  {}_n\!\Br'(X)  \lra 0
$$
where  $\Pic(X) = H_{\et}^1(X, \GG_{m, X})$. 
%$\Pic(X)_n = {\rm{Coker}}\{ \Pic(X)  \overset{\alpha \ra \alpha.n}\lra \Pic(X) \}$and ${}_n\!\Br'(X) = {\rm{Ker}}\{ \Br'(X) \overset{\alpha \ra \alpha.n}\lra \Br'(X)\}$. 
We have a similar  exact sequence for $X_K$, which gives rise to the following commutative diagram 

\begin{equation}
\begin{CD}
0 @>>> \Pic(X)_n@>>> H_{\et}^2(X, \mu_{n, X})@>>> {}_n\!\Br'(X)@>>> 0\\
@. @VVV @VVV @VVV\\
0 @>>> \Pic(X_K)_n@>>> H_{\et}^2(X_K, \mu_{n, X_K})@>>> {}_n\!\Br'(X_K)@>>> 0
\end{CD}
\end{equation}

The map $  H_{\et}^2(X, \mu_{n, X})  \ra H_{\et}^2(X_K, \mu_{n, X_K})$ is an isomorphism by the proper base change theorem  \cite[Exp. XII,  Corollary 5.4]{SGA4}. On another hand, for every geometric point $\bar x \ra X_K$,  we have by  \cite[Proposition 5.3]{ES}
  $\pi_1^\et(X_K, \bar{x}) \simeq \pi_1^\et(X, \bar{x}) $. It follows from Lemma \ref{h1}.(a) that
\begin{align}
{}_n\!\Pic(X_K) & \simeq    H_{\et}^1(X_K, \mu_{n, X_K})  \notag  \\
& \simeq  H^1(\pi_1^\et(X_K, \bar{x}), (\mu_{n, X_K})_{\bar x}) \notag \\
& \simeq  H^1(\pi_1^\et(X, \bar{x}), (\mu_{n, X})_{\bar x}) \notag \\
 & \simeq    H_{\et}^1(X, \mu_{n, X})  \notag  \\
 & \simeq    {}_n\!\Pic(X) \notag  
\end{align}
Therefore,  from the above diagram we conclude that $\Br'(X)= \Br'(X_{K})$.
 \qed

\medskip
Descent of Brauer maps in Proposition \ref{desbr},  can be extended to schemes over subfields of $\CC$ by the properness condition which implies the compactness  of $X^{an}$ in addition to Proposition  \ref{propbr}. We have then the following result.

\begin{proposition}\label{prpr}  Let $X$ be a  proper,  geometrically connected  scheme of finite type over  an algebraically closed field $k$. Suppose that $k$ can be embedded as a subfield of $\CC$ 
 and such that 
\begin{itemize}
\item [(i)] $X_\CC$ is an Artin neighborhood over $\Spec(\CC)$.

\item [(ii)]  The natural morphism $\Br(X) \lra  \Br(X_{\CC})$ is surjective.
\end{itemize}
Then $\Br(X)= \Br'(X)$.

\end{proposition}

\proof  Consider the following commutative diagram
\begin{equation}
\begin{CD}
 \Br(X) @>>> \Br(X_\CC) \\
@VVV @VVV \\
 \Br'(X) @>>> \Br'(X_\CC) 
\end{CD}
\end{equation}
The  map on the right is an isomorphism by Corollary \ref{propb}. By Proposition \ref{propbr} the lower map is also an isomorphism. The injectivity of Brauer maps (Theorem \ref{groth}.(iii))  implies  that $\Br(X) \ra  \Br(X_{\CC})$ is injective,  hence bijective. Thus the assertion.
\qed

\medskip

\section{Smooth case: Descent of Artin neighborhoods}\label{S5}
Le $X$ be a smooth scheme over a field $k$,  then $X$ is in particular a  regular scheme,  and hence  geometrically unibranch. Keeping in mind Proportion \ref{kphom},  we use descent properties of $K(\pi, 1)$ spaces  to apply Theorem \ref{rcnb}. We need the following proposition,  which is a first descent result concerning  Artin neighborhoods. The notations $\pi_n^{\top}$ and $\pi_n^\et$ will be used to make difference between topological and \'etale homotopy groups.
\begin{proposition}
\label{artkp}
Let  $X$ be  a   smooth  connected scheme of finite type  over $\CC$,  if X is an Artin neighborhood over $\Spec(\CC)$, then   $X$ is an algebraic $K(\pi, 1)$ space.
\end{proposition}

\begin{proof}[First proof] 
Let $x \in X(\CC)$,  and $\bar x \ra X$ the geometric point image of $x$ under the natural map  $p: X^{an} \ra X_\et$.  By Lemma \ref{khtp} $X(\CC)$ is weakly homotopy equivalent  to  the classifying space $B\pi_1^{\top}(X(\CC), x)$. Since $X$ is smooth,  hence geometrically unibranch,  by \cite[Corollary 12.10]{AM} the map  $ (X(\CC))^{\wedge} \rightarrow Et(X) $ is an $\natural$-isomorphism (cf. \cite[Definition 4.2]{AM}),  where $X^{\wedge}$ is the completion of $X$ (see \cite[Theorem 3.4]{AM}). Thus,  according to  \cite[Corollary 4.4]{AM}   $Et(X)$ is weakly homotopy equivalent to $(X(\CC))^{\wedge}$. Since $\pi_1^{\rm top}(X(\CC), x)$ is a good group  by Lemma \ref{khtp},  it follows from    \cite[Corollary 6.6]{AM} that $(B\pi_1^{\rm top}(X(\CC), x))^ {\wedge} =  B (\pi_1^{\rm top}(X(\CC), x)^ {\wedge})$. On the other hand, by  Riemann existence theorem \cite[Chapter III,  Lemma 3.14]{MIL} one has $\pi_1^\et(X, \bar{x}) = \pi_1^{\top}(X(\CC), x)^ {\wedge} $. Therefore, $Et(X)$  is weakly homotopy equivalent to $B\pi_1^\et(X, \bar{x})$, which means that  $\pi_n^\et(X, \bar{x}) = 0$ for all $n \geq 2$. Thus by  Proposition \ref{kphom}  $X$ is an algebraic $K(\pi, 1)$ space.
 \end{proof}

\begin{proof}[Second proof] Alternatively, one can deduce the statement  from the following commutative diagram 

$$
\begin{CD}
 H^q(\pi_1^\et(X, \bar{x}),  \mathcal{F}_{\bar x})     @>>>   H^q(\pi_1^{\top}(X(\CC), x),   (p^{*}\mathcal{F})_x)  \\
@VVV @VVV\\
H_{\et}^q(X, \mathcal{F})    @>>> H^q(X(\CC), p^{*}\mathcal{F})  
\end{CD}
$$
where $\shF$ is a  locally constant constructible torsion \'etale sheaf on $X$.  Indeed, the  map on the right  is an isomorphism because $X(\CC)$ is a topological $K(\pi, 1)$ space.  The lower map  is an isomorphism by Artin comparison theorem \cite[Exp. XVI, Theorem 4.1]{SGA4}(see also the smooth version \cite[Exp. XI, Theorem 4.4]{SGA4}),  and since  $\pi_1^\et(X, \bar{x}) = \pi_1^{\top}(X(\CC), x)^ {\wedge}$, and  $\mathcal{F}_{\bar x}= (p^{*}\mathcal{F})_x$,  it follows from the definition of good groups that the upper map is an isomorphism. Hence the assertion.
 \end{proof}
\begin{proof}[Third proof] By definition there exists a sequence of elementary fibrations 
$$X=X_i  \overset{f_i} \lra X_{i-1}  \overset{f_{i-1}} \lra... \overset{f_2} \lra X_1  \overset{f_1} \lra  X_0=\Spec(\CC)$$
Since $\Spec(\CC)$ is a $k(\pi,1)$ space in the sens of  \cite{SCHM}, it suffices to prove that for an elementary fibration $f: X \ra Y$ and $\bar{y} \ra Y$  a geometric point of $Y$, if  $\pi_n^\et(Y,  \bar y) = 0$ for $n \geq 2$   then we have the same for $X$. Let $f: X \ra Y$ be an elementary fibration, if  $\bar{x} \ra X$ and   $\bar{y} \ra Y$  are geometric points with $ \bar y = f(\bar x)$, then by   \cite[Theorem 11.5]{EMF}  there exists an exact sequence of  \'etale homotopy groups  
$$
...   \ra \pi_n^\et(X_{\bar{y}},  \bar x )   \ra  \pi_n^\et(X,  \bar x)  \ra  \pi_n^\et(Y,  \bar y) \ra  \pi_{n-1}^\et(X_{\bar{y}},  \bar x )  \ra ...
$$
On the other hand, it follows from the definition of elementary fibrations that $X_{\bar{y}}$ is a smooth affine curve, and hence by Example \ref{expkp} $  \pi_n^\et(X_{\bar{y}},  \bar x ) = 0$ for all $n \geq 2$, the assertion follows immediately. 
\end{proof}

\medskip
\begin{corollary}\label{artinc}
Let  $X$ be  a   smooth connected scheme of finite type  over $\CC$,  if X is an Artin neighborhood over $\Spec(\CC)$, then   $\Br(X)= \Br'(X)$.
\end{corollary}
\proof By Proposition \ref{artkp}  $X$ is an algebraic $K(\pi, 1)$ space. Proportion \ref{kphom} shows that $\pi_n^\et(X, \bar{x})=0$ for all $n \geq 2$, thus Theorem \ref{rcnb}  implies that $\Br(X)= \Br'(X)$.
\qed

\medskip
\begin{proposition}\label{smooth}
 Let  $X$ be a  smooth,  geometrically connected  scheme of finite type over a field k. Suppose that $k$ can be embedded as a subfield of $\CC$ which is finitely generated over $\QQ$
,  and such that  $X_\CC$ is an Artin neighborhood over $\Spec(\CC)$. Then $\Br(X)= \Br'(X)$.
\end{proposition}

\proof  Let  $p: X_{\CC} \ra  X_{\bar{k}}$  be the natural map and $\bar x \ra X_{\CC}$ a geometric point of $X_{\CC}$. By Theorem \ref{rcnb} it suffices to prove that $\pi_2^\et(X, \bar{x}) = 0$. For every locally constant constructible torsion  \'etale sheaf $ \mathcal{F}$  on $X_{\bar{k}} $ and $q \geq 0$  we have the following commutative diagram 
$$
\begin{CD}
 H^q(\pi_1^\et(X_{\bar{k}}, \bar{x}),  \mathcal{F}_{\bar{x}})     @>>>   H^q(\pi_1^\et(X_{\CC}, \bar{x}),  \mathcal{F}_{\bar{x}})  \\
@VVV @VVV\\
H_{\et}^q(X_{\bar{k}}, \mathcal{F})    @>>> H_{\et}^q(X_{\CC}, p^{*}\mathcal{F})  
\end{CD}
$$
According to   \cite[Proposition 6.1]{ES}  we have  $\pi_1^\et(X_{\bar{k}}, \bar{x}) \simeq \pi_1^\et(X_{\CC}, \bar{x}) $,  hence    the map $$H^q(\pi_1^\et(X_{\bar{k}}, \bar{x}),  \mathcal{F}_{\bar{x}}) \lra   H^q(\pi_1^\et(X_{\CC}, \bar{x}),  \mathcal{F}_{\bar{x}})$$  is an isomorphism. The map $H_{\et}^q(X_{\bar{k}}, \mathcal{F})  \rightarrow    H_{\et}^q(X_{\CC}, p^{*}\mathcal{F})   $ is an isomorphism by the smooth base change theorem \cite[Exp. XVI,  Corollary 1.6]{SGA4}. By Proposition \ref{artkp} $X_{\CC}$ is an algebraic $K(\pi, 1)$ space,  which means that
 the map $ H^q(\pi_1^\et(X_{\CC}, \bar{x}),  \mathcal{F}_{\bar{x}})  \rightarrow   H_{\et}^q(X_{\CC}, p^{*}\mathcal{F}) $ is an isomorphism. It follows that  
 $$ H^q(\pi_1^\et(X_{\bar{k}}, \bar{x}),  \mathcal{F}_{\bar{x}}) \lra H_{\et}^q(X_{\bar{k}}, \mathcal{F})$$ 
is an isomorphism. Thus $ X_{\bar{k}}$ is an algebraic $K(\pi,  1)$ space. Next, consider the Grothendieck  homotopy exact sequence (cf. \cite[Exp. IX,  Theorem 6.1]{SGA1})
$$
1 \lra \pi_1^\et(X_{\bar{k}}, \bar x )   \lra  \pi_1^\et(X,  \bar x )  \lra  \pi_1^\et(\Spec(k),  \bar x ) \lra 1
$$
The term on the right is just the absolute Galois group $\Gal( \bar k |k)$. The injectivity of the map $\pi_1^\et(X_{\bar{k}}, \bar x )   \ra  \pi_1^\et(X,  \bar x )$ means that any finite \'etale cover of $X_{\bar{k}}$ can be realized as the restriction to $X_{\bar{k}}$ of a finite \'etale cover of $X$. Now suppose that   $ \mathcal{F}$ is a  locally constant constructible torsion  \'etale sheaf   on $X$. Take a class $\beta \in  H_{\et}^2(X, \shF)$, and consider the Leray spectral sequence associated to the morphism $X \ra \Spec(k)$ 
$$
E^{p, q}_2 = H_{\et}^p(k, H_{\et}^q(X_{\bar k}, \rho^*\shF)) \Rightarrow   H_{\et}^{p+q}(X, \shF)
$$
where $\rho: X_{\bar k} \ra X$ is the natural morphism.  The   cohomology group $E^2 =  H_{\et}^2(X, \shF)$ has a filtration $F_n(E^2)_{n\geq 0}$ with three highest subquotients,  which are in fact submodules of $E^{0, 2}$,  $ E^{1, 1}$ and $E^{2, 0}$ respectively
\begin{equation}
E^2 = \left\{ \begin{array}{rcl}
E^{0+2} \thicksim   \Gr_0(E^{0+2})= F^0(E^2)/ F^1(E^2)   \simeq E^{0, 2}_\infty  \subset E^{0, 2}_2 = H_{\et}^0(k, H_{\et}^2(X_{\bar k}, \rho^*\shF))\\
E^{1+1}  \thicksim   \Gr_1(E^{1+1})= F^1(E^2)/ F^2(E^2)   \simeq E^{1, 1}_\infty  \subset E^{1, 1}_2 = H_{\et}^1(k, H_{\et}^1(X_{\bar k}, \rho^*\shF))\\
E^{2+0}  \thicksim   \Gr_2(E^{2+0})= F^2(E^2)/ F^3(E^2)  \simeq E^{2, 0}_\infty  \subset E^{2, 0}_2 =  H_{\et}^2(k, H_{\et}^0(X_{\bar k}, \rho^*\shF))

\end{array}\right.
\end{equation}
Therefore, we may assume that the class $\beta$ belongs to one of these three submodules.

If $\beta \in H_{\et}^0(k, H_{\et}^2(X_{\bar k}, \rho^*\shF)) =  H_{\et}^2(X_{\bar k}, \rho^*\shF) $,  then since $X_{\bar k}$ is a $K(\pi, 1)$ space,  there exists a finite \'etale cover  $g:Y \ra X_{\bar k}$ trivializing $\beta$ in  $H_{\et}^2(Y, g^{*}\rho^*\shF)$. As mentioned above there exists a finite \'etale cover $f:Y' \ra X$ such that $Y = Y' \times_X X_{\bar k}$. Hence we may assume that $f^{*}(\beta)$ is zero in   $H_{\et}^2(Y', f^{*}\shF)$.

If $\beta \in  H_{\et}^1(k, H_{\et}^1(X_{\bar k}, \rho^*\shF)) $,  it follows from Proposition \ref{kpcover} and Lemma \ref{h1}  (or particularly  form the fact that $\Spec(k)$ is an algebraic $K(\pi, 1)$ space)  that there exists a finite \'etale cover $g:Y \ra \Spec(k)$ trivializing $\beta$. Therefore  if we consider the finite \'etale  cover $f: Y'=X \times_k Y  \ra X$,  then we may assume that $f^{*}(\beta) = 0 $    in   $H_{\et}^2(Y', f^{*}\shF)$.

If $\beta \in  H_{\et}^2(k, H_{\et}^0(X_{\bar k}, \rho^*\shF)) $,  by the same argument we may assume the existence of a finite \'etale cover killing $ \beta$.

The up-shot is that for every class $\beta \in H_{\et}^2(X,  \shF)$ there exists a finite \'etale cover $f:Y \ra X$ such that that $f^{*}(\beta) = 0 $    in   $H_{\et}^2(Y, f^{*}\shF)$,  which means by Remark \ref{kpdep}  that the map  $ H^2(\pi_1^\et(X, \bar{x}),  \mathcal{F}_{\bar{x}}) \ra  H_{\et}^2(X, \mathcal{F})$ is an isomorphism, and hence $\pi_2^\et(X, \bar x) =0$.
\qed

\medskip
The base change arguments in the smooth case can also serve in the proper case, provided the given scheme is regular. 
\begin{proposition}\label{prprreg}  Let $X$ be a  regular  proper  geometrically connected  scheme of finite type over  an algebraically closed field $k$. Suppose that $k$ can be embedded as a subfield of $\CC$ 
 and such that $X_\CC$ is an Artin neighborhood over $\Spec(\CC)$. 
Then $\Br(X)= \Br'(X)$. 
\end{proposition}

\proof

By Theorem \ref{rcnb} and Proposition \ref{kphom}  it suffices to prove that  $X$ is an algebraic $K(\pi, 1)$ space. As in the smooth case, let  $p: X_{\CC} \ra  X$  be the natural map and $\bar x \ra X_{\CC}$ a geometric point of $X_{\CC}$.  For every locally constant constructible torsion  \'etale sheaf $ \mathcal{F}$  on $X$ and $q \geq 0$  we have a commutative diagram 
$$
\begin{CD}
 H^q(\pi_1^\et(X, \bar{x}),  \mathcal{F}_{\bar{x}})     @>>>   H^q(\pi_1^\et(X_{\CC}, \bar{x}),  \mathcal{F}_{\bar{x}})  \\
@VVV @VVV\\
H_{\et}^q(X, \mathcal{F})    @>>> H_{\et}^q(X_{\CC}, p^{*}\mathcal{F})  
\end{CD}
$$
Due to    \cite[Proposition 5.3]{ES}  we have  $\pi_1^\et(X, \bar{x}) \simeq \pi_1^\et(X_{\CC}, \bar{x}) $,  hence    the map $$H^q(\pi_1^\et(X, \bar{x}),  \mathcal{F}_{\bar{x}}) \lra   H^q(\pi_1^\et(X_{\CC}, \bar{x}),  \mathcal{F}_{\bar{x}})$$  is an isomorphism. The proper  base change theorem \cite[Exp. XII,  Corollary 5.4]{SGA4} asserts that  $H_{\et}^q(X, \mathcal{F})  \simeq    H_{\et}^q(X_{\CC}, p^{*}\mathcal{F})$. The map  $ H^q(\pi_1^\et(X_{\CC}, \bar{x}),  \mathcal{F}_{\bar{x}})  \rightarrow   H_{\et}^q(X_{\CC}, p^{*}\mathcal{F}) $ is an isomorphism by Proposition \ref{artkp}. Hence the assertion.
\qed

\medskip
Now if we want to extend the statement to algebraically closed fields of characteristic 0,  we  consider  the following assumption for a scheme $X$ over a field $k$
\begin{center}
$(H)$: \{ If  $k \subseteq \CC$,  then $X_{\CC}$ is an Artin neighborhood over $\Spec(\CC)$\}
\end{center}
\begin{proposition}\label{alg0}  
Let  $X$ be a   smooth, proper   geometrically connected  scheme of finite type over a field k. Suppose that X satisfies (H). Then $\Br(X)= \Br'(X)$ when $k$ is an algebraically closed field of characteristic 0.
\end{proposition}
\proof
Since $k$ is   algebraically  closed  of characteristic 0,   there is a  subfield $ F \subset k$ finitely generated over $\QQ$ and $X$ is defined over $F$,  that is there exists  a proper,  smooth,   geometrically connected scheme $Y$ of finite type  over $F$ such that $Y_k = X$. Choose an embedding $i: \bar F \ra \CC$, and let $\bar y \ra Y_{\CC}$ be a geometric point of $Y_{\CC}$. On the one hand,  since $Y$ satisfies $(H)$,  then $Y_{\CC}$ is an Artin neighborhood over $\Spec(\CC)$. Hence by Proposition \ref{artkp} and Proposition \ref{kphom} $\pi_n^\et(Y_{\CC}, \bar y) =0$ for all $n\geq 2$. On the other hand, since $Y$ is proper,  it follows from   \cite[Corollary 12.12]{AM} and  \cite[Corollary 4.4]{AM} that  $ Et(Y_{\bar F})$ is weakly homotopy equivalent to $Et(Y_\CC)$ and  $ Et(Y_k)= Et(X)$  is weakly homotopy equivalent to $ Et(Y_{\bar F})$, hence $\pi_n^\et(X, \bar x) =0$ for all $n\geq 2$, where $\bar x$ is a geometric point above $\bar y$. Now  Theorem \ref{rcnb} applies.
\qed

\medskip

More generally,  let's  replace the morphism $ X \ra \Spec(k) $ by a morphism $f: X \ra S$ of connected noetherian schemes. If $f$ is flat proper,  with geometrically connected and reduced fibers, and $\bar s \ra S$ a  geometric point of $S$, then for a  geometric point  $\bar{x} \ra X_{\bar{s}}$  in the fiber $f^{-1}(\bar{s}) =X_{\bar{s}}$ above $\bar{s}$,  we have an exact sequence (cf. \cite[Exp. X,  Corollary 1.4]{SGA1})
$$
\pi_1^\et(X_{\bar{s}},  \bar x )   \lra  \pi_1^\et(X,  \bar x)  \lra  \pi_1^\et(S,  \bar s) \lra 1
$$
Suppose now that we are in a situation under which the map on the left is injective. For example,  let  $X$ and $S$ be as above, and let   $Y \subset X$ be a complement of a normal crossing divisor relative to S. Let $g: Y \ra S$ be the restriction map,  and   $\bar{y} \ra Y_{\bar{s}}$ a geometric point  with $g(\bar y) = \bar s$. Suppose moreover that $f: X \ra S$ is smooth and admits a section,  then there is  an homotopy exact sequence (see \cite[Exp. XIII,  Proposition 4.1 and Example 4.4]{SGA1})
$$
1 \lra \pi_1^\et(Y_{\bar{s}},  \bar y )   \lra  \pi_1^\et(Y,  \bar y)  \lra  \pi_1^\et(S,  \bar s) \lra 1
$$
If we consider the  Leray spectral sequence associated to the morphism $g: Y \ra S$ 
$$
E^{p, q}_2 = H_{\et}^p(S, H_{\et}^q(Y_{\bar{s}}, p^*\shF)) \Rightarrow   H_{\et}^{p+q}(Y, \shF)  
$$ 
where $\shF$ is a  locally constant constructible torsion  \'etale sheaf  on $Y$, and $p: Y_{\bar{s}} \ra Y$ the natural map, then the  same argument used in the last step in the proof of Proposition  \ref{smooth}  can be applied  to  get the following.
\begin{proposition}\label{overS}
Let $f: X \ra S$ and $Y \subset X$ are as above with X  regular  in characteristic 0. If $S$ and $Y_{\bar s}$ are algebraic $K(\pi, 1)$ spaces, then $\Br(Y)= \Br'(Y)$.
\end{proposition}

\begin{remark}
As in the third proof of  Proposition \ref{artkp}, another special case of the situation discussed above  is when  $f: X \ra S$ is an elementary fibration, then for geometric points   $\bar{x} \ra X$ and   $\bar{s} \ra S$   with $ \bar s = f(\bar x)$ there is an exact sequence of \'etale  homotopy groups  
$$
... \ra \pi_2^\et(S,  \bar s)   \ra \pi_1^\et(X_{\bar{s}},  \bar x )   \ra  \pi_1^\et(X,  \bar x)  \ra  \pi_1^\et(S,  \bar s) \ra  \pi_0^\et(X_{\bar{s}},  \bar x)  
$$
If $S$ is an  algebraic $K(\pi, 1)$ space, then since $X_{\bar s}$ is also an  algebraic $K(\pi, 1)$ space by Example \ref{expkp}, and $\pi_0^\et(X_{\bar{s}},  \bar x)=0$, we get an exact sequence of \'etale fundamental groups 
$$
1 \lra \pi_1^\et(X_{\bar{s}},  \bar x )   \lra  \pi_1^\et(X,  \bar x)  \lra  \pi_1^\et(S,  \bar s) \lra  1
$$
Therefore, by the previous argument one can alternatively show that $\Br(X)= \Br'(X)$.

\end{remark}

\section{Local $K(\pi, 1)$ condition}\label{S6}

Grothendieck proved \cite[II,  Theorem 2.1]{GR1} that for a noetherian  scheme $X$,  and $\beta \in \Br'(X)$,  there exists an open $U \subseteq X$ such that  $X - U$  has codimension   $\geq 2$,  and an Azumaya algebra $\shA$ on $U$ such that $\delta([\shA])= \beta_{|U}$. He applied this to show that for a regular noetherian scheme of dimension $\leq 2$,  one has $\Br(X) = \Br'(X)$ \cite[II,   Corollary 2.2]{GR1}. In the next theorem we prove that the same statement holds for a smooth $k$-variety of arbitrary dimension,  provided that the subscheme $U$ is an algebraic $K(\pi, 1)$ space. This assumption is enhanced by Artin theorem \cite[Exp XI,  Proposition  3.3]{SGA4} by which  any smooth scheme over an algebraically closed field of characteristic 0  has a cover by  Zariski   opens $k(\pi, 1)$ subschemes. Furthermore,   in  the light of Achinger generalization of Artin result \cite{ACH2},  we can choose $k$ of positive characteristic. A key ingredient in the proof is the following purity theorem for the cohomological Brauer group which was  predicted by Grothendieck  in \cite[III, Section 6]{GR1} and   proved   recently by \v Cesnavi\v cius in the general case.
\begin{theorem}\cite[Theorem 6.1]{CES}
\label{purity}
Let  $X$ be a regular noetherian scheme,   and   $  U \subset X$  an open subscheme  such that the complement $X - U$  has codimension   $\geq 2$. Then $ H_{\et}^2(X, \GG_{m, X}) \simeq  H_{\et}^2(U, \GG_{m, U})$.

%geometrically connected  scheme of finite type 

\end{theorem}
Combined with  Grothendieck result, this theorem provides the following local statement for Brauer groups.

\begin{proposition}
For any regular noetherian scheme $X$,  there exists an open $U \subset X$ with codimension of $X - U$   $\geq 2$ such that $\Br(U) = \Br'(U)$.
\end{proposition}

Now we assert   the main theorem.

\begin{theorem}\label{localbr}
Let  $X$ be a  smooth variety over an algebraically closed field $k$ of characteristic $p \geq 0$,   such that any pair of point $(x, y) \in X $  is contained in an affine open scheme. Suppose that there exists  an algebraic $K(\pi, 1)$ open subscheme    $   Y \subset X$   such that  $X - Y$  has codimension   $\geq 2$. Then $\Br(X)= \Br'(X)$ up to a $p$-component.
\end{theorem}

\proof

Fix an integer $n$  prime to $p$,  and choose  a class $\beta \in H_{\et}^2(X, \mu_{n, X})$. Let $\bar y \ra Y$ be a geometric point of $Y$. As in the proof of Proposition \ref{propbr} the Kummer exact sequence for both $X$ and $Y$ gives rise to the following  commutative diagram 
\begin{equation}
\begin{CD}
0 @>>> \Pic(X)_n@>>> H_{\et}^2(X, \mu_{n, X})@>>> {}_n\!\Br'(X)@>>> 0\\
@. @VVV @VVV @VVV\\
0 @>>> \Pic(Y)_n@>>> H_{\et}^2(Y, \mu_{n, Y})@>>> {}_n\!\Br'(Y)@>>> 0.
\end{CD}
\end{equation}
On the one hand,  the map ${}_n\!\Br'(X) \ra {}_n\!\Br'(Y)$ is an isomorphism by Theorem \ref{purity}. On the other hand by Zariski-Nagata purity theorem,  the functor $S \ra S \times_X Y$ induces an equivalence of categories  between finite \'etale covers of $X$ and finite \'etale covers of $Y$,  thus $\pi_1^\et(X, \bar y) \simeq \pi_1^\et(Y, \bar y)$ (cf. \cite[Exp X,  Theorem 3.10]{SGA2}). Applying Lemma \ref{h1}.(a)  we get the following  isomorphisms
$${}_n\!\Pic(X) \simeq H_{\et}^1(X, \mu_{n, X}) \simeq  H^1(\pi_1^\et(X, \bar y), (\mu_{n, X})_{\bar y})$$
$${}_n\!\Pic(Y) \simeq  H^1(\pi_1^\et(Y, \bar y), (\mu_{n, Y})_{\bar y})$$
%Here $\bar x \ra X$ is a geometric point of $X$ with $\bar x = i(\bar y)$, where $i: Y \ra X$ is the inclusion map. 
Hence the map $\Pic(X)_n \ra \Pic(Y)_n$ is bijective. It follows from the above commutative diagram that 
$H_{\et}^2(X, \mu_{n, X}) \simeq H_{\et}^2(Y, \mu_{n, Y})$. 
By assumption, $Y$ is an algebraic $K(\pi, 1)$ space, which means that  the map $$H^2(\pi_1^\et(Y,  \bar y ),  (\mu_{n, Y})_{\bar y}) \lra   H_{\et}^2(Y, \mu_{n, Y})$$ is an isomorphism. Therefore,  from  the following commutative diagram 
$$
\begin{CD}
 H^2(\pi_1^\et(X, \bar y ),  (\mu_{n, X})_{\bar y})     @>>>   H^2(\pi_1^\et(Y, \bar y),  (\mu_{n, Y})_{\bar y})  \\
@VVV @VVV\\
H_{\et}^2(X, \mu_{n, X})    @>>> H_{\et}^2(Y, \mu_{n, Y})  
\end{CD}
$$
we get an isomorphism
$$ H^2(\pi_1^\et(X, \bar y ),  (\mu_{n, X})_{\bar y})  \simeq H_{\et}^2(X, \mu_{n, X})$$
Due to Proposition \ref{kpcover} there exists a finite \'etale cover $f: X' \ra X$ such that $f^{*}(\beta)=0$ in $H_{\et}^2(X', \mu_{n, X'})$. Therefore,  in the light of Lemma   \ref{etale} and Remark  \ref{h2cech}   we conclude that $\Br(X)= \Br'(X)$.
\qed

\medskip

\section{Application to abelian varieties}\label{S7}

In \cite{HOO1}(see also \cite{HOO2}) Hoobler showed  that $\Br(A) = \Br'(A) $   for an abelian variety $A$ over a field $k$   by proving that $A$ satisfies the generalized theorem of cube. Recall (see \cite[Section 2]{HOO1}) that an abelian scheme $A$ over a noetherian scheme $S$ satisfies the generalized theorem of cube for $l$,  if the natural morphism
\[
\begin{matrix}
\prod  s_{ij}^{*}:  & H_{\et}^2(A^3, \GG_{m, A^3})(l)  & \lra &  (H_{\et}^2(A^2, \GG_{m, A^2}(l)) ^3\\
 & x &  \lra & (s_{12}^{*} (x),  s_{13}^{*} (x), s_{23}^{*} (x))
\end{matrix}
\]
is injective, where $l$ is a prime 
distinct from the residue characteristics of $A$, and  $s_{ij} : A \times_S A  \ra A \times_S A  \times_S A  $ is the map 
which inserts the unit section $	S \ra A$ into the $k$-th factor for $k \in \{1,2,3\}-\{i,j\}$. Note that  this notion  was extended by Bertolin and Galluzzi    to 1-motives  (see \cite[Definition 5.1]{BERT}).

An alternative  proof was given by Berkovich in \cite{BERK} where he showed that for an abelian variety over a separably closed field $k$ and $n$ is prime to $char(k)$  one has
$$H_{\et}^2(A, \mu_{n, A}) \simeq  \bigwedge^2 \Hom(A[n],   \mu_{n, A} )$$
 Therefore, if $k$ is an arbitrary field and $\alpha \in \Br'(A) $ with $n\alpha = 0$ and  $n$  prime to $char(k)$,  the composition  map $\pi: A_{\bar k} \overset{n_A }\lra A_{\bar k}  \overset{i} \lra A$ is an \'etale Galois cover with  $\pi^{*}(\alpha) = n^2.\alpha = 0$ in $H_{\et}^2(A_{\bar k}, \mu_{n, A_{\bar k}})$,  where $ A_{\bar k} \overset{n_A}\lra A_{\bar k}$ is the multiplication by $n$ and $ A_{\bar k}  \overset{i} \lra A$ the natural map.

We give in turn  another proof  based on the \'etale homotopy type of abelian varieties. We need the following result of   Demarche and Szamuely,  which is a general form of the Riemann existence theorem for smooth connected algebraic groups. To simplify notations, we omit the geometric base points.
\begin{lemma}
%\cite[Proposition 2.2]{DMR}
\label{riemalg}  Let G be a connected smooth algebraic group over
$\CC$. For all   $n \geq 1 $ there is an isomorphism
 $$\pi^{\et}_n(G) \simeq \pi_n^{\rm top}(G(\CC))^{\wedge}$$
\end{lemma}
\proof (Sketch). As in the first proof of  Proposition \ref{artkp}  $(G(\CC))^{\wedge}$ is weakly homotopy equivalent to $ET(G)$. On the other hand, Demarche and Szamuely remarked that the homotopy groups $\pi_n^{\rm top}(G(\CC))$ are finitely generated abelian groups. Thus by a result of Sullivan \cite[Theorem 3.1]{DS} the natural map $\pi_n^{\rm top}(G(\CC)) ^{\wedge} \ra \pi_n^{\rm top}(G(\CC)^{\wedge})$ is an isomorphism. Hence the assertion.
\qed

\medskip
\begin{theorem}
\label{abel}
Let $A$ be an abelian variety over a field $k$ of characteristic 0. Then $\Br(A) = \Br'(A)$.
\end{theorem}
\proof
Note that by a limit argument  \cite[Corollary 4]{HOO2} we can assume $k$  algebraically closed. Hence   there is a  subfield $ F \subset k$ finitely generated over $\QQ$ and $A$ is defined over $F$,  that is there is  an abelian variety $B$  over $F$ such that $B_k = A$. Choose an embedding $i: \bar F \ra \CC$. Applying   the above lemma,  we get $\pi_n^{\et}(B_\CC) \simeq \pi_n^{\rm top}(B_{\CC}(\CC))^{\wedge}$. Since $B$ is proper,  it follows from   \cite[Corollary 12.12]{AM} and  \cite[Corollary 4.4]{AM} that  $ Et(B_{\bar F})$ is weakly homotopy equivalent to $Et(B_\CC)$ and  $ Et(B_k)= Et(A)$  is weakly homotopy equivalent to $ Et(B_{\bar F})$.  Therefore $\pi^{\et}_n(A)=\pi_n^{\rm top}(B_{\CC}(\CC))^{\wedge}=0$ for all $n \geq 2$,  because $B_{\CC}(\CC)$ is a complex tori which is a topological $K(\pi, 1)$ space. Hence Theorem \ref{rcnb} implies that $\Br(A) = \Br'(A)$. Alternatively,  since $A$ is geometrically unibranch,  by Proposition \ref{kphom}  it is an algebraic $K(\pi, 1)$ space. Since any finite set of points of $A$ is contained in an affine open,  the statement results then from Proposition \ref{kp1finit}.  
\qed

\medskip

\begin{remark}
The argument sketched above  shows in fact that if $G$ is a  smooth geometrically connected algebraic group over an algebraically  closed field $k$ of characteristic 0,  such that  $G(\CC)$  is a topological $K(\pi, 1)$ space in the case that $k=\CC$,  then  $G$ is an algebraic $K(\pi, 1)$ space. This generalize the result  of \cite[Corollary 5.5.(b)]{PAL} for  an abelian variety  $A$ which requires the goodness of $\pi^{\rm top}_n(A(\CC))$. Furthermore,  one can use  \cite[Proposition 5.4]{PAL}  to prove that $\Br'(G) = \Br(G)$  which does not require the properness assumption.

\end{remark}

\section{Pro-universal covers}\label{S8}
Let $X$ be a quasi-compact and quasi-separated connected scheme,  and consider $A = (X_i,  f_{ij})$  the projective system of finite \'etale  covers of $X$. For every element $f_i: X_i \ra X$ in $A_{ij}$ there exists an \'etale  Galois cover $g_i: Y_i \ra X$ which factors through $X_i$,  hence elements in this projective system can be taken to be Galois. Since all transition maps   $f_{ij}: X_i \rightarrow X_j$  are affine,  then by \cite[Exp VII,  5.1]{SGA4}    the projective limit  $\hat{X}:= \varprojlim X_i$ exists as a scheme. The cover $\hat f: \hat{X} \ra X$ is called the pro-universal cover of $X$. Moreover,  by \cite[Chapter III,  Lemma 1.16]{MIL}  for every \'etale sheaf $\mathcal{F}$ on $X$ and for any $q \geq 0$ we have 
$$
H_{\et}^q( \hat{X},  {\hat f}^{*}\mathcal{F})  \simeq  \varinjlim H_{\et}^q( X_i,  f_{i}^{*}\shF)  
$$
On another hand, for a given geometric point $\bar x \ra X$  we have 
$$
   \pi_1^\et(X, \bar{x})  \simeq  \varprojlim \Aut_X(X_i) 
$$
where $ \Aut_X(X_i) $ is the group of $X$-automorphisms of $X_i$ acting on the right.

The next proposition shows that the cohomology of the pro-universal cover $\hat X$ has a direct consequence on the surjectivity of the Brauer map. The proof is based on Hochschild-Serre theory and arguments similar the the ones in the proof of Theorem \ref{rcnb}. We first state the following lemma.

\begin{lemma}\label{h1pro}
\mylabel{vanish1}
Let  $X$   be a connected noetherian scheme,  then for any locally constant  constructible torsion \'etale sheaf $\shF$ on $X$, we have  $H_{\et}^1( \hat{X},  {\hat f}^{*}\shF ) = 0 $.

\end{lemma}
\proof 

We will prove that for every finite \'etale cover $f: Y \ra X$ and for every class $\beta \in  H_{\et}^1( Y,  f^*\mathcal{F})$ there exists a finite \'etale cover $h: Z \ra X$  such that  $h=fog$ for some finite \'etale cover $g: Z \ra Y$,  and  $(h_{|Y})^*(\beta):=g^{*}(\beta) = 0$ in $ H_{\et}^1( Z,  h^*\mathcal{F})$. Choose  a geometric point $\bar x \ra X$, and let $f: (Y,\bar y)  \ra (X, \bar x) $  be a pointed  finite \'etale cover. By  Lemma \ref{h1}.(a)   we have
$$H_{\et}^1(Y,  f^*\mathcal{F}) \simeq   H^1(\pi_1^\et(Y, \bar y),  (f^*\mathcal{F})_{\bar y}) \simeq  H_{\fet}^1(Y,  f^*\mathcal{F})$$ 
 Hence there exists a   finite \'etale cover $g: Z \ra Y$ such that $g^{*}(\beta) = 0$ in $ H_{\et}^2( Z,  g^{*} f^{*}\mathcal{F})$. Put $h = f \circ g: Z \ra X$,  it is a finite \'etale cover of $X$  with $(h_{|Y})^*(\beta):=g^{*}(\beta) = 0$.
\qed

\medskip

\begin{proposition}\label{probr}
Let $X$  be a regular connected  scheme of finite type over a field k of characteristic 0. Suppose that $H_{\et}^2( \hat{X},  {\hat f}^{*}\shF)= 0$ for every locally constant constructible torsion \'etale sheaf $ \mathcal{F}$ on $X$. Then $\Br(X)= \Br'(X)$.

\end{proposition}

\proof

Choose  a geometric point $\bar x \ra X$. For every \'etale Galois cover $f_i:  X_i \rightarrow X$, we consider the Hochschild-Serre spectral sequence 
$$
E^{p, q}_2(X_i)= H^p(\Aut_X( X_i),    H_{\et}^q(X_i,  f_{i}^{*}\shF )) \Rightarrow  H_{\et}^{p+q}(X, \mathcal{F})
$$
Taking the inductive limit  of $E^{p, q}_2(X_i)$ we get by \cite[Chapter I, \S 2.2,  Proposition 8]{SER} 
\begin{align}
\varinjlim H^p(\Aut_X( X_i),   H_{\et}^q(X_i,  f_{i}^{*}\shF)  ) & \simeq   H^p(\varprojlim \Aut_X( X_i),  \varinjlim H_{\et}^q(X_i,  f_{i}^{*}\shF)  ) \notag \\
& \simeq  H^p(\pi_1^\et(X, \bar{x}),   H_{\et}^q( \hat{X},   {\hat f}^{*}\shF) )\notag
\end{align}
Hence  we obtain  by \cite[Chapter III,  Remark 2.21.b]{MIL} a spectral sequence 
$$
E^{p, q}_2 =  H^p(\pi_1^\et(X, \bar{x}),   H_{\et}^q( \hat{X},  {\hat f}^{*}\shF))   \Rightarrow  H_{\et}^{p+q}(X, \mathcal{F})
$$
which yields  an exact sequence of law-degrees terms
\begin{align}
0 \ra  H^1(\pi_1^\et(X, \bar{x}),   H_{\et}^0( \hat{X},  {\hat f}^{*}\shF))    \ra   H_{\et}^1(X, \mathcal{F})  & \ra   H^0(\pi_1^\et(X, \bar{x}),   H_{\et}^1( \hat{X},  {\hat f}^{*}\shF))\\
 \ra  H^2(\pi_1^\et(X, \bar{x}),   H_{\et}^0( \hat{X},  {\hat f}^{*}\shF))   \ra  H_{\et}^2(X, \mathcal{F}) & \ra   H^0(\pi_1^\et(X, \bar{x}),   H_{\et}^2( \hat{X},  {\hat f}^{*}\shF )) \notag
\end{align}
By assumption $H_{\et}^2( \hat{X},  {\hat f}^{*}\shF )= 0$,  and by Lemma \ref{vanish1} $H_{\et}^1( \hat{X},  {\hat f}^{*}\shF) = 0$,  hence we get an isomorphism
$$
 H_{\et}^2(X, \mathcal{F})  \simeq H^2(\pi_1^\et(X, \bar{x}),   H_{\et}^0( \hat{X},  {\hat f}^{*}\shF ) ) 
$$
Now take $\mathcal{F} = \mu_{n,X}$ for some integer $n$,  and put $F:=  H_{\et}^0( \hat{X},  {\hat f}^{*}\shF )$. Then we have 
\begin{align}
H_{\et}^2(X,  \mu_{n,X}) & \simeq   H^2(\pi_1^\et(X, \bar{x}),  F) \notag  \\
& \simeq \varinjlim H^2(\pi_1^\et(X, \bar{x})/ N,  F^N) \notag 
\end{align}
 where the limit  runs over all open normal subgroups $N$ of $\pi_1^\et(X, \bar{x})$,  and $F^N$ is the submodule of $N$-invariant elements. Next, choose a class  $\beta \in  H_{\et}^2(X,  \mu_{n,X}) $. Note that if $\rho: Z \ra X$ is a finite \'etale cover,  it is easy to see  that $H_{\et}^2( \hat{Z},  {\hat h}^{*}(\rho^{*}\shF) )= 0$,
 where ${\hat h} :{\hat Z} \ra Z$ is the pro-universal cover of $Z$. Hence proceeding as in the proof of Theorem \ref{rcnb},  we can find an \'etale Galois cover $g: Y \ra X$ killing $\beta$ in $H_{\et}^2( Y, \mu_{n,Y})$. Thus by Lemma \ref{galois}  $\Br(X)= \Br'(X)$.
\qed

\medskip

We finish this section by a result concerning  smooth quasi-projective varieties, by which condition proposed above turns out to be equivalent to that in Theorem \ref{rcnb}, and hence equivalent  to the $K(\pi,1)$ propriety  for $H^2$ by Remark \ref{kpdep}. A key ingredient in the proof is the following interpretation of $\pi_2^\et(X, \bar{x})$ in terms of \'etale homology via the \'etale Hurewicz map.

Following P\'al \cite{PAL}, for every abelian group $A$ and $n \geq 0$ we  consider the homology groups
$$
H_n(X, A):= H_n( Et(X),  A)
$$
and the \'etale Hurewicz maps
$$
h_n(X, \bar{x}) :   \pi_n^\et(X, \bar{x})  \lra    H_n( Et(X), \hat{ \ZZ})
$$
where $\hat  \ZZ =  \varprojlim \ZZ /n\ZZ$ is the profinite completion of $\ZZ$. We have  the following interpretation of $\pi_2^\et(X, \bar{x}) $.

\begin{proposition}\cite[Theorem 4.3]{PAL}
\label{hrwtz}
  Let $(X,  \bar x)$  be a pointed  smooth,  geometrically irreducible,  quasi-projective variety  over an algebraically closed  field k,  then the \'etale Hurewicz map $h_2(X, \bar{x}) $ yields  an isomorphism 
$$
   \pi_2^\et(X, \bar{x})  \simeq  \varprojlim H_2( Et(X_i), \hat{ \ZZ})
$$
where the limit runs over all finite \'etale  covers  $(X_i \rightarrow X)$ of $X$.

\end{proposition}

\begin{proposition}\label{carpi}
 Let  $(X, \bar x)$  be a pointed smooth,  geometrically irreducible,  quasi-projective variety over an algebraically closed field  k. Then  the following   conditions are equivalent 
\begin{itemize}

\item [(i)]  $\pi_2^\et(X, \bar x) = 0$.
\item [(ii)]   $H_{\et}^2( \hat{X},  {\hat f}^{*}\shF )= 0$ for every  locally constant constructible torsion \'etale 
 sheaf $\mathcal{F}$ on $X$.
\item [(iii)]   $H_{\et}^2(X,  \mathcal{F}) \simeq   H^2(\pi_1^\et(X, \bar x),  \mathcal{F}_{\bar x})$ for every  locally constant constructible torsion \'etale 
 sheaf $\mathcal{F}$ on $X$.
\end{itemize}

\end{proposition}

\proof 
(i)$\Leftrightarrow$(iii): This holds by Remark \ref{kpdep}. Alternatively, the implication (i)$\Rightarrow$(iii) can be deduced directly from the exact sequence 
$$
0\lra  H^2(\pi_1^\et(X, \bar{x}),   \mathcal{F}_{\bar x})    \lra  H_{\et}^2(X, \mathcal{F})   \lra \Hom( \pi_2^\et(X, \bar{x}),  \mathcal{F}_{\bar x})^{\pi_1^\et(X, \bar{x})} 
$$
(iii)$\Rightarrow$(ii): we have  to prove that for every finite \'etale cover $f: Y \ra X$ and for every class $\beta \in  H_{\et}^2( Y,  f^*\mathcal{F})$ there exists a finite \'etale cover $h: Z \ra X$  factors through $f$, i.e. $h=fog$ for some finite \'etale cover $g: Z \ra Y$,  such that $(h_{|Y})^*(\beta):=g^{*}(\beta) = 0$ in $ H_{\et}^2( Z,  h^*\mathcal{F})$. Let $f: (Y,\bar y)  \ra (X, \bar x) $  be a pointed  finite \'etale cover. We have  by Lemma \ref{h1}.(b)   $H_{\et}^2(Y,  f^*\mathcal{F}) \simeq   H^2(\pi_1^\et(Y, \bar y),  (f^*\mathcal{F})_{\bar y})$, hence for  a class  $\beta \in  H_{\et}^2( Y,  f^*\mathcal{F})$  Proposition \ref{kpcover} implies that   there exists a   finite \'etale cover $g: Z \ra Y$ such that $g^{*}(\beta) = 0$ in $ H_{\et}^2( Z,  g^{*} f^{*}\mathcal{F})$. Put $h = f \circ g: Z \ra X$,  it is a finite \'etale cover of $X$  with  $ h^{*}= g^{*}f^{*} $ and $(h_{|Y})^*(\beta):=g^{*}(\beta) = 0$.

(ii)$\Rightarrow$(i).   Let $ \mathcal F = \hat \ZZ$   we have 
\begin{equation}
H_{\et}^2( \hat{X},  {\hat f}^{*}\shF ) = H_{\et}^2( \hat{X}, \hat{ \ZZ}) = H^2( Et(\hat{X}), \hat{ \ZZ})
\end{equation}
For all $n \geq 1$,we have by  the universal coefficient theorem for cohomology  an exact sequence
$$
0  \lra  \Ext^1_\ZZ( H_{n-1}(\hat{X},  \ZZ),  \hat{\ZZ})  \lra  H_{\et}^n( \hat{X}, \hat{ \ZZ})  \lra  \Hom_\ZZ(H_n(\hat{X},  \ZZ), \hat{ \ZZ})  \lra 0
$$
Thus for $n=1,2$ it  follows from assumption and Lemma \ref{h1pro} that $H_1(\hat{X},  \ZZ)=0$ and $H_2(\hat{X},  \ZZ)=0$. Now  the universal coefficient theorem for homology yields the following exact sequence
$$
0  \lra   H_2(\hat{X},  \ZZ) \otimes \hat \ZZ  \lra  H_2( \hat{X}, \hat{ \ZZ})  \lra  \Tor_1(H_1(\hat{X},  \ZZ), \hat{ \ZZ})  \lra 0
$$
Thus  $H_2( \hat{X}, \hat{ \ZZ})=0$. Since $\varinjlim H_{\et}^2( X_i,  \hat{ \ZZ}) \simeq  H_{\et}^2(\hat{X},  \hat{ \ZZ})$, it follows from  Proposition \ref{hrwtz} and homological Yoneda lemma \cite[Lemma 1.1]{PVLT} that
$$
\pi_2^\et(X,\bar x ) \simeq \varprojlim H_2( X_i, \hat{ \ZZ}) \simeq H_2( \hat{X}, \hat{ \ZZ})=0 
$$
\qed

\begin{acknowledgments}I would like to thank my advisor  Rachid Chibloun for his continuous encouragement  throughout the stages of this work. This paper is dedicated to Professor Raymond Hoobler, who sadly passed away from complications of COVID-19.  
\end{acknowledgments}

\end{document}